\documentclass[%draft,
a4paper]
{amsart}

\usepackage[all,cmtip]{xy}
\usepackage{amssymb}
\usepackage[usenames,dvipsnames]{xcolor}
\usepackage[utf8]{inputenc}
\usepackage{tikz}
\usetikzlibrary{matrix,arrows,calc,fit,positioning}
\usepackage{graphicx}
\usepackage{mathrsfs}  
\usepackage{todonotes}

\newtheorem{theorem}[equation]{Theorem}
\newtheorem{corollary}[equation]{Corollary}
\newtheorem{lemma}[equation]{Lemma}
\newtheorem{proposition}[equation]{Proposition}
\theoremstyle{definition}
\newtheorem{definition}[equation]{Definition}
\theoremstyle{remark}
\newtheorem{remark}[equation]{Remark}

\numberwithin{equation}{section}

\newcommand{\abs}[1]{|#1|}

\def\To{\longrightarrow}
\def\mono{\hookrightarrow}
\def\into{\rightarrowtail}
\def\onto{\twoheadrightarrow}
\def\hom{\operatorname{Hom}}

\def\st{\stackrel} % abreviatura de \stackrel
\def\cplx{\operatorname{Ch}}
\def\bcplx{\operatorname{bCh}}
\def\tcplx{\operatorname{tCh}}

\def\tot{\operatorname{Tot}}
\def\CE{\operatorname{CE}}

\newcommand{\lradjunction}{\,\,\raisebox{-0.1\height}{$\overrightarrow{\longleftarrow}$}\,\,}

\begin{document}

\title%[shot title]
{Homotopy theory of bicomplexes}%
\author{Fernando Muro}%
\address{Universidad de Sevilla,
Facultad de Matemáticas,
Departamento de Álgebra,
Avda. Reina Mercedes s/n,
41012 Sevilla, Spain}
\email{fmuro@us.es}
\urladdr{http://personal.us.es/fmuro}

\author{Constanze Roitzheim}%
\address{University of Kent, School of Mathematics, Statistics and Actuarial Science, Sibson Building, Canterbury, CT2 7FS, UK}
\email{csrr@kent.ac.uk}
\urladdr{http://www.kent.ac.uk/smsas/personal/csrr}

\thanks{The first author was partially supported by the Spanish Ministry of Economy under the grant MTM2016-76453-C2-1-P (AEI/FEDER, UE)}
\subjclass%[2010]
{}

\begin{abstract}
We define two model structures on the category of bicomplexes concentrated in the right half plane. The first model structure has weak equivalences detected by the totalisation functor. The second model structure's weak equivalences are detected by the $E^2$-term of the spectral sequence associated to the filtration of the total complex by the horizontal degree. We then extend this result to twisted complexes.
\end{abstract}

\maketitle

\tableofcontents

% ----------------------------------------------------------------

\section*{Introduction}

The notion of chain complex is central to homological algebra, as they arise e.g. as resolutions of modules. Bicomplexes, in turn, arise as resolutions of chain complexes. These resolutions, first defined by Cartan and Eilenberg \cite{cartan_homological_1956}, are concentrated in a half-plane. They are used to compute derived functors between derived categories \cite{Gelfand_Manin_2003} in conjuction with the totalisation functor from bicomplexes to complexes. 

\bigskip
Homotopy theory of chain complexes is a well-known concept. Model categories provide a commonly used language to construct resolutions in the presence of a notion of equivalence weaker than isomorphisms. One such example is homology isomorphisms of chain complexes, also known as quasi-isomorphisms. A standard model structure on the category of chain complexes of modules over a ring $k$ is the \emph{projective} model structure, where weak equivalences are quasi-isomorphisms and cofibrant objects are also degreewise projective. % And indeed, it can be shown that the standard definition of being chain homotopic is equivalent to being homotopic with regards to this model structure. 

\bigskip
The first goal of this paper is to provide useful model structures on the category of \emph{bicomplexes} $X_{*,*}$ concentrated in the right half plane, i.e.~$X_{p,q}=0$ for $p<0$. In the first one, the \emph{total model structure}, the weak equivalences are exactly those morphisms whose totalisation induces an isomorphism in homology. In the second model structure, the \emph{Cartan--Eilenberg model structure}, cofibrant resolutions are Cartan--Eilenberg resolutions. Weak equivalences in this second model structure are the so-called ``$E^2$-equivalences''. Bicomplexes have horizontal and vertical differentials,
\[d_h\colon X_{p,q}\longrightarrow X_{p-1,q},\qquad d_v\colon X_{p,q}\longrightarrow X_{p,q-1}.\]
These differentials anti-commute, i.e.
$$d_h d_v + d_v d_h = 0.$$ This condition implies that the horizontal differential $d_h$ induces a differential on the vertical homology $H^v_*(X)$ of a bicomplex $X$, so one can consider
$$ H^h_*(H^v_*(X)).$$ This is exactly the $E^2$-term of a spectral sequence that converges strongly to the homology of the total complex $\tot(X)$ of the bicomplex $X$. Therefore, the ($H^h_* \circ H^v_*$)-isomorphisms are called $E^2$-equivalences, and they are special cases of $(H_* \circ \tot)$-isomorphisms.

\bigskip
Both model structures will enjoy useful properties such as being combinatorial, proper, and monoidal. Furthermore, we will see that the first one is an \emph{abelian} model structure, i.e.~fibrations are surjections with fibrant kernel and cofibrations are injections with cofibrant cokernel. Cofibrant objects are also degreewise projective in both cases, so we can really see our model structures as generalisations of the projective model structure on chain complexes for different choices of weak equivalences. The total model structure, in addition, is Quillen equivalent to the model category of chain complexes.

\bigskip
We will then generalise the total model structure on bicomplexes to the category of \emph{twisted complexes}. While a bicomplex is a bigraded $k$-module with two differentials, a twisted complex is equipped with maps $$d_i: X_{p,q} \longrightarrow X_{p-i, q+i-1}, \qquad i \ge 0,$$ satisfying $$\sum\limits_{i+j=n} d_i d_j=0, \qquad n \ge 0.$$ Naturally, making all the necessary definitions and calculations to obtain the total model structure on this category is a lot more involved than in the case of just two differentials. 

\bigskip
Another motivation for these model structures stems from the study of $A_\infty$-algebras, or ``homotopy associative'' algebras. Among other things, $A_\infty$-structures on the homology of a differential graded algebra allow us to see how many differential graded algebras realise this homology. However, this only works over a ground field \cite{Kadeishvili_1980}, or if all modules in question are projective. To circumnavigate this rather restrictive assumption, one can work in the context of \emph{derived $A_\infty$-algebras} \cite{sagave_dg-algebras_2010}. These are bigraded objects, where the second degree allows to create a projective resolution compatible with any $A_\infty$-structure. Where $A_\infty$-algebras have an underlying chain complex, derived $A_\infty$-algebras have an underlying twisted chain complex concentrated in the right half plane. 
Furthermore, the homological perturbation lemma \cite{Brown_1967} tells us that the vertical homology of every Cartan--Eilenberg resolution can be equipped with the structure of a twisted complex.

Therefore, in order to understand the homotopy theory of derived $A_\infty$-algebras, specifically in an operadic context \cite{livernet-roitzheim-whitehouse, cirici-egas-livernet-whitehouse}, it is necessary to understand the homotopy theory of the underlying twisted complexes.

\bigskip

This paper is organised as follows. In Section \ref{sec:complexes} we recall some basic definitions and results concerning chain complexes and the projective model structure, in particular on how the model structure is constructed using spheres and discs. In Section \ref{sec:bicomplexes} we study the category $\bcplx$ of bicomplexes concentrated in the right half-plane, define the bigraded analogue of spheres and discs and discuss the tensor product. Sections \ref{sec:total} and \ref{sec:ce} give the total and the Cartan--Eilenberg model structures by showing that the generating cofibrations and trivial cofibrations defined using those spheres and discs together with the respective weak equivalences satisfy Smith's recognition principle from \cite[Theorem 2.1.19]{hovey_model_1999}. Finally, we introduce the category of twisted complexes in Section \ref{sec:twisted}, define spheres and discs, and obtain the desired model structure. 

\section{Complexes}\label{sec:complexes}

We briefly recall a couple of facts about the model categories $\cplx$ of unbounded (chain) complexes and $\cplx_{\geq 0}$ of complexes concentrated in non-negative degrees. We use the convention that differentials shift the degree by $-1$. 

Throughout this paper, $k$ denotes a commutative ground ring. Further conditions on $k$ will be imposed when necessary. Tensor product will always be taken over $k$.

As a category, $\cplx$  is locally finitely presentable. Limits and colimits are computed pointwise. It is also a closed symmetric monoidal category with respect to the tensor product. The symmetry constraint uses the Koszul sign convention, and the inner $\hom$ is the graded module
\[\hom_{\cplx}(X,Y)_n=\prod_{m\in\mathbb Z}\hom_k(X_m,Y_{m+n})\]
endowed with the following differential
\[d(f)=df-(-1)^{\abs{f}}fd.\]
Here $\hom_k$ denotes the inner $\hom$ in the category of modules.

The category $\cplx$ is also a combinatorial proper model category. Weak equivalences are quasi-isomorphisms and fibrations are (pointwise) surjections. Let us recall the generating (trivial) cofibrations. For a $k$-module $A$, we define the chain complex $D^n(A)$ to be 
\[\cdots\to 0\to A\st{1}\to A\to 0\to\cdots \]
concentrated in degrees $n$ and $n-1$. Similarily we define $S^{n}(A)$ to just consist of $A$ concentrated in degree $n$. This in fact gives us adjoint functor pairs
\[
ev_n: k\mbox{-mod} \lradjunction \cplx: D^n
\]
and 
\[
Z_n: k\mbox{-mod} \lradjunction \cplx: S^n
\]
where $ev_n$ denotes evaluation at degree $n$ and $Z_n(X)=\ker[d\colon X_{n}\to X_{n-1}]$ denotes the cycles in degree $n$. (Note that when we write adjunctions, the top arrow is always the left adjoint.)

We now define the \emph{$n$-sphere} $S^{n}$ to be $S^{n}(k)$ and the \emph{$n$-disk} $D^n$ to be $D^n(k)$ for short. This is used to construct the \emph{projective model structure} on $\cplx$, which is the model structure we will consider throughout this paper. Define sets $I_{\cplx}$ and $J_{\cplx}$ as
\begin{align*}
	I_{\cplx}&=\{S^{n-1}\mono D^n\}_{n\in\mathbb Z},\\
	J_{\cplx}&=\{0\mono D^n\}_{n\in\mathbb Z}.
\end{align*}
Here $S^{n-1}\mono D^n$ is the identity in degree $n-1$. Furthermore, let $\mathcal{W}$ denote the class of $H_*$-isomorphisms, i.e.~quasi-isomorphisms. 

Recall that for a class of maps $I$, the class $I$-inj is given by all the maps that have the right lifting property with respect to $I$. Furthermore, $I$-cof is the class of all maps that have the left lifting properties with respect to all maps in $I$-inj. The class $I$-cell is given by all transfinite compositions of pushouts of elements of $I$. This class satisfies $I\text{-cell}
\subseteq I\text{-cof}$. 

Then in our case $I_{\cplx}$, $J_{\cplx}$, and $\mathcal{W}$ satisfy the following properties:
\begin{itemize}
\item $\mathcal{W}$ has the two-out-of-three property and is closed under retracts,
%\item the domains of $I$ are small relative to $I$-cell
%\item the domains of $J$ are small relative to $J$-cell
\item $J_{\cplx}\mbox{-cell}\subseteq \mathcal{W} \cap I_{\cplx}\mbox{-cof}$,
\item $I_{\cplx}\mbox{-inj}\subseteq \mathcal{W} \cap J_{\cplx}\mbox{-inj}$,
\item either $ \mathcal{W} \cap I_{\cplx}\mbox{-cof} \subseteq J_{\cplx}\mbox{-cof}$ or $\mathcal{W}\cap J_{\cplx}\mbox{-inj} \subseteq I_{\cplx}\mbox{-inj}$ (a posteriori both),
\end{itemize}
plus some set-theoretic conditions. By \cite[Theorem 2.1.19]{hovey_model_1999}, this means that there is a cofibrantly generated model structure with weak equivalences $\mathcal{W}$, generating cofibrations $I_{\cplx}$, and generating trivial cofibrations $J_{\cplx}$. (Trivial) fibrations are the maps in $J_{\cplx}\mbox{-inj}$ (resp.~$I_{\cplx}\mbox{-inj}$), and cofibrations are retracts of maps in $J_{\cplx}\mbox{-cell}$. It can furthermore be shown using the adjunctions defined earlier, that a map is a fibration if and only if it is a degreewise surjection, and that cofibrations are the degreewise monomorphisms with cofibrant cokernel.
The last property means that this model structure is \emph{abelian} in the sense of \cite{hovey_cotorsion_2007}, i.e.~a cofibration is a monomorphism with cofibrant cokernel and a fibration is an epimorphism with fibrant kernel. Cofibrant objects do not have an easy characterization, but they are known to be pointwise projective. %, but bounded below complexes of projective modules are known to be cofibrant.

Furthermore, it is compatible with the monoidal structure in the sense of \cite[Definition 4.2.6]{hovey_model_1999}. The tensor unit $S^0$ is actually cofibrant since it is the cokernel of the generating cofibration $S^{-1}\mono D^0$. The monoid axiom of Schwede and Shipley  \cite[Definition 3.3]{schwede_algebras_2000} is also satisfied.

The full subcategory $\cplx_{\geq 0}\subset\cplx$ of complexes concentrated in non-negative degrees inherits a monoidal model structure with the same tensor product and weak equivalences. The inner $\hom_{\cplx_{\geq 0}}(X,Y)$ is the non-negative truncation of $\hom_{\cplx}(X,Y)$. The former is a subcomplex of the latter, both complexes coincide in (strictly) positive degrees, and 
\[\hom_{\cplx_{\geq 0}}(X,Y)_0=\cplx_{\geq 0}(X,Y)=Z^0(\hom_{\cplx}(X,Y)).\]
The fibrations in $\cplx_{\geq 0}$ are the maps which are surjective on positive degrees but not necessarily in degree $0$. Cofibrations are precisely the maps with pointwise projective cokernel. Sets of generating (trivial) cofibrations are
\begin{align*}
I_{\cplx_{\geq 0}}&=\{0\mono S^0\}\cup\{S^{n-1}\mono D^n\}_{n\geq 1},\\
J_{\cplx_{\geq 0}}&=\{0\mono D^n\}_{n\geq 1}.
\end{align*}
The model structure on $\cplx_{\geq 0}$ is proper, monoidal, with cofibrant monoidal tensor unit $S^0$, and it satisfies the monoid axiom. It is not abelian, though, since fibrations need not be surjective. The inclusion $\cplx_{\geq 0}\subset\cplx$ is a left Quillen functor. Its right adjoint is the non-negative truncation.

We will use the outline of these well-known results as a blueprint for the model structures on bicomplexes, respectively twisted chain complexes, that we are going to construct in the subsequent chapters.

\section{Bicomplexes}\label{sec:bicomplexes}

This section consists of elementary definitions and examples which are relevant for later computations. We consider $(\mathbb N\times\mathbb Z)$-graded bicomplexes made of anticommutative squares.

\begin{definition}\label{bicomplexes}
	A \emph{bicomplex} $X$ is a bigraded module $X=\{X_{p,q}\}_{p,q\in\mathbb Z}$ with $X_{p,q}=0$ for $p<0$, equipped with \emph{horizontal} and \emph{vertical differentials}
	\[d_h\colon X_{p,q}\To X_{p-1,q}, \qquad 
	d_v\colon X_{p,q}\To X_{p,q-1},\]
	respectively, satisfying
	\[d_vd_h+d_hd_v=0.\]
	A morphism of bicomplexes $f\colon X\to Y$ is a family of maps $f_{p,q}\colon X_{p,q}\to Y_{p,q}$ compatible with the horizontal and vertical differentials
	\[d_hf=fd_h,\qquad d_vf=fd_v.\]
	We denote the category of bicomplexes by $\bcplx$.
	
	For any bigraded module $X$, given $x\in X_{p,q}$, we say that $\abs{x}_h=p$ is the \emph{horizontal degree} of $x$, and $\abs{x}_v=q$ is its \emph{vertical degree}. The \emph{bidegree} of $x$ is $(\abs{x}_h,\abs{x}_v)=(p,q)$ and the \emph{total degree} is $\abs{x}=\abs{x}_h+\abs{x}_v=p+q.$
\end{definition}

The category $\bcplx$ is clearly abelian and locally finitely presentable. Limits and colimits are computed pointwise.

\begin{remark}\label{commuting}
	The equation $d_vd_h+d_hd_v$ says that the following squares are anticommutative in a bicomplex
	\[\xymatrix{X_{p-1,q}\ar[d]_{d_v}&X_{p,q}\ar[d]^{d_v}\ar[l]_-{d_h}\\
		X_{p-1,q-1}&X_{p,q-1}\ar[l]^-{d_h}}\]
	Some readers will probably prefer that these squares commute. If we denote by $X'$ the underlying bigraded module of $X$ endowed with the same horizontal differential $d_h'=d_h$ and the new vertical differential $d_v'$ defined by
	\[d_v'(x)=(-1)^{\abs{x}_h}x,\]
	we obtain a bicomplex $X'$ with commuting differentials $d_h'd_v'=d_v'd_h'$. If $\bcplx'$ denotes the category of bicomplexes with commuting differentials, we obtain an isomorphims of categories
	\begin{align*}
		\bcplx&\cong\bcplx'\\X&\mapsto X'.
	\end{align*}
\end{remark}

\begin{definition}
	Let $X$ be a bicomplex. The \emph{vertical cycles} $Z^v(X)$ are the elements in the kernel of the vertical differential of $X$, and the \emph{vertical boundaries} $B^v(X)$ are the elements in the image of $d_h$. The \emph{vertical homology} is \[H^v(X)=\frac{Z^v(X)}{B^v(X)}.\] 
	
	We can regard $Z^v(X)$, $B^v(X)$, and $H^v(X)$ as bicomplexes with trivial vertical differential. Their vertical differentials are induced by that of $X$. We similarly define the \emph{horizontal cycles} $Z^h(X)$, \emph{boundaries} $B^h(X)$, and \emph{homology} $H^h(X)$.
	
\medskip	
Let $A$ be a $k$-module. Then we define $D^{p,q}(A)$, $p> 0$, $q\in\mathbb Z$, to be the bicomplex whose underlying bigraded module is $$D^{p,q}_{p,q}(A)=D^{p,q}_{p-1,q}(A)=D^{p,q}_{p,q-1}(A)=D^{p,q}_{p-1,q-1}(A)=A$$ and zero elsewhere. Its four nontrivial differentials are given by the identity except for $$d_v=-1: D^{p,q}(A)_{p,q} \longrightarrow D^{p,q}(A)_{p,q-1}.$$ 
Furthermore, we define $\partial_hD^{p,q}(A)$ and $\partial_vD^{p,q}(A)$ to be the horizontal resp.~vertical boundaries of $D^{p,q}(A)$. We also define $S^{p,q}(A)$ to be the bicomplex with $A$ in bidegree $(p,q)$ and zero in all other degrees. 

\medskip
The following can be easily verified.

\begin{lemma}\label{representation}
The above definitions give rise to adjunctions
\begin{align*}
ev_{p,q}  &\colon   k\mbox{-mod} \lradjunction \bcplx \colon  D^{p,q} \\
Z^h_{p-1,q} &\colon   k\mbox{-mod} \lradjunction \bcplx \colon  \partial_h D^{p,q} \\
Z^v_{p-1,q} &\colon   k\mbox{-mod} \lradjunction \bcplx \colon  \partial_v D^{p,q} \\
Z^h_{p} \circ Z^v_q  &\colon   k\mbox{-mod} \lradjunction \bcplx \colon  S^{p,q} 
\end{align*}
Here $ev_{p,q}$ denotes evaluation at bidegree $(p,q)$.
\end{lemma}

\medskip	
	We define the \emph{$(p,q)$-disc} $D^{p,q}$ as $D^{p,q}(k)$. We can view it as the bicomplex freely generated by a single element $x_{p,q}\in D^{p,q}$.  The free $k$-module generators are $x_{p,q}$, $d_h(x_{p,q})$, $-d_v(x_{p,q})$, and $d_vd_h(x_{p,q})=-d_hd_v(x_{p,q})$, respectively.
	The \emph{horizontal} and \emph{vertical boundaries} of the $(p,q)$-disk will again be denoted by $\partial_hD^{p,q}$ and $\partial_vD^{p,q}$, respectively. The \emph{$(p,q)$-sphere} is $k$ concentrated in bidegree $(p,q)$. These bicomplexes look as follows,
	\begin{center}
		\begin{tikzpicture}[execute at begin node=$, execute at end node=$]
		\draw[->, thick, opacity =.2] (3,1) -- (8,1);
		\draw[->, thick, opacity =.2] (3,0) -- (3,4);
		\node at (2.7,3) {\scriptscriptstyle q};
		\draw[thick, opacity=.2] (2.9,3) -- (3.1,3);
		\node at (2.55,2) {\scriptscriptstyle q-1};
		\draw[thick, opacity=.2] (2.9,2) -- (3.1,2);
		\node at (5,3) (E) {k};
		\node at (5,2) (F) {k};
		\node at (6,3) (C) {k};
		\node at (6,2) (D) {k};
		\draw[->] (C) -- node [right] {\scriptscriptstyle -1} node [left] {\scriptstyle D^{p,q}} (D);
		\draw[->] (E) -- (F);
		\draw[->] (C) -- (E);
		\draw[->] (D) -- (F);
		\node[fit=(D)(F), draw, inner sep=-.9,rounded corners=5] {};
		\node[fit=(E)(F), draw, inner sep=-.9,rounded corners=5] {};
		\node[fit=(F), draw, inner sep=-.9,rounded corners=5] {};
		\node at (6.8,2) {\scriptstyle \partial_vD^{p,q}};
		\node at (5,3.4) {\scriptstyle \partial_hD^{p,q}};
		\node at (4.5,1.5) {\scriptstyle S^{p-1,q-1}};
		\draw[thick, opacity=.2] (5,.9) -- (5,1.1);
		\node at (5,.7) {\scriptscriptstyle p-1};
		\draw[thick, opacity=.2] (6,.9) -- (6,1.1);
		\node at (6,.7) {\scriptscriptstyle p};
		\end{tikzpicture}
	\end{center}
	Here, unlabelled arrows are identities. 
\end{definition}

\begin{remark}
	The bicomplex $\partial_hD^{p,q}$ is freely generated by the horizontal cycle $d_h(x_{p,q})$ in bidegree $(p-1,q)$. Similarly, $\partial_vD^{p,q}$ is freely generated by the vertical cycle $-d_v(x_{p,q})$ in bidegree $(p,q-1)$. Since our bicomplexes are concentrated in non-negative horizontal degree, $\partial_hD^{1,q}$ is freely generated by the element $d_h(x_{1,q})$ in bidegree $(0,q)$. Indeed, morally, we can define $D^{0,q}=\partial_hD^{1,q}$ and $\partial_vD^{0,q}=S^{0,q-1}$. The reader can check that most of the properties of $D^{p,q}$ and $\partial_vD^{p,q}$ for $p>0$ extend to the case $p=0$ with these definitions, but we have preferred two avoid two different notations for the same object.
\end{remark}

As a consequence of Lemma \ref{representation}, we have the following useful natural isomorphisms, which we list for convenience. 

\begin{corollary}\label{cor:adjunctions}
	For any bicomplex $X$, $p>0$, and $q\in\mathbb Z$, there are natural isomorphisms
	\begin{align*}
		\bcplx(S^{0,q-1}\mono \partial_hD^{1,q},X)&\cong
		\cplx(S^{q-1}\mono D^{q},X_{0,*}),\\
		\bcplx(\partial_vD^{p,q}\mono D^{p,q},X)&\cong
		\cplx(S^{q-1}\mono D^{q},X_{p,*}),\\
		\bcplx(0\mono \partial_hD^{1,q},X)&\cong
		\cplx(0\mono D^{q},X_{0,*})\\
		&\cong
		\cplx(0\mono S^{0},X_{*,q}),\\
		\bcplx(0\mono S^{0,q-1},X)&\cong
		\cplx_{\geq0}(0\mono S^0,Z^v_{*,q-1}(X)),\\
		\bcplx(S^{p-1,q-1}\mono\partial_vD^{p,q},X)&\cong
		\cplx_{\geq0}(S^{p-1}\mono D^p,Z^v_{*,q-1}(X)),\\
		\bcplx(\partial_hD^{p,q}\mono D^{p,q},X)&\cong
		\cplx_{\geq 0}(S^{p-1}\mono D^{p},X_{*,q}),\\
		\bcplx(0\mono\partial_vD^{p,q},X)&\cong
		\cplx_{\geq0}(0\mono D^p,Z^v_{*,q-1}(X)).				
	\end{align*}	
\end{corollary}
\qed

We now consider the monoidal structure on bicomplexes.

\begin{definition}\label{monoidal_bicomplex}
	The \emph{tensor product} $X\otimes Y$ of two bicomplexes $X$ and $Y$ is the bicomplex defined as
	\[(X\otimes Y)_{p,q}=\bigoplus_{\substack{m+s=p\\n+t=q}}X_{m,n}\otimes Y_{s,t}\]
	with horizontal and vertical differentials defined as
	\begin{align*}
		d_h(x\otimes y)&=d_h(x)\otimes y+(-1)^{\abs{x}}x\otimes d_h(y),\\
		d_v(x\otimes y)&=d_v(x)\otimes y+(-1)^{\abs{x}}x\otimes d_v(y).
	\end{align*}
	Note that both formulas use the total degree in their sign conventions.
	
	This tensor product endows $\bcplx$ with a closed symmetric monoidal structure with obvious associativity and unit constraints. The tensor unit is $k$ concentrated in bidegree $(0,0)$. The symmetry constraint uses the Koszul sign rule with respect to the total degree,
	\begin{align*}
		X\otimes Y&\cong Y\otimes X,\\
		x\otimes y&\mapsto (-1)^{\abs{x}\abs{y}}y\otimes x.
	\end{align*}
	The mapping objects $\hom_{\bcplx}(X,Y)$ in $\bcplx$, adjoints to the tensor product, are defined by the $k$-modules
	\[\hom_{\bcplx}(X,Y)_{p,q}= \prod_{\substack{s\geq 0\\t\in\mathbb Z}}\hom_k(X_{s,t},Y_{s+p,t+q}),\qquad p> 0, \quad q\in\mathbb Z,\]
	and the submodules
	\[\hom_{\bcplx}(X,Y)_{0,q}\subset \prod_{\substack{s\geq 0\\t\in\mathbb Z}}\hom_k(X_{s,t},Y_{s,t+q}),\qquad q\in\mathbb Z,\]
	formed by the elements $f$ such that
	\[d_hf=(-1)^{\abs{f}}fd_h.\]
	The horizontal and vertical differentials are defined by
	\[d_h(f)=d_hf-(-1)^{\abs{f}}fd_h,\qquad d_v(f)=d_vf-(-1)^{\abs{f}}fd_v.\]
\end{definition}

\begin{remark}
	The previous definition would not work for bicomplexes with commuting differentials (see Remark \ref{commuting}). The readers which prefer commuting differentials  will probably find more natural to consider the horizontal and vertical degrees in the definition of the horizontal and vertical differentials of the tensor product. Indeed, this yields a bicomplex $X'\otimes'Y'$ with commuting differentials
	\begin{align*}
	d_h'(x'\otimes y')&=d_h'(x')\otimes y'+(-1)^{\abs{x'}_h}x'\otimes d_h'(y'),\\
	d_v'(x'\otimes y')&=d_v'(x')\otimes y'+(-1)^{\abs{x'}_v}x'\otimes d_v'(y').
	\end{align*}
	The underlying bigraded module of $X'\otimes'Y'$ is obviously defined as for $X\otimes Y$.
	
	In this case it ismore sensible to use the Koszul sign rule with respect to the horizontal and vertical degrees separately in the definition of the symmetry constraint, 
	\begin{align*}
	X'\otimes' Y'&\cong Y'\otimes' X',\\
	x'\otimes y'&\mapsto (-1)^{\abs{x'}_h\abs{y'}_h+\abs{x'}_v\abs{y'}_v}y'\otimes x'.
	\end{align*}
	This endows the category $\bcplx'$ of bicomplexes with commuting differentials with a closed symmetric monoidal structure.
	
	The isomorphism of categories $\bcplx\cong\bcplx'$ in Remark \ref{commuting} together
	with the natural isomorphism
	\begin{align*}
	(X\otimes Y)'&\cong X'\otimes' Y',\\
	x\otimes y&\mapsto (-1)^{\abs{x}_v\abs{y}_h}x\otimes y,
	\end{align*}
	defines a symmetric monoidal isomorphism.
	We will work with $\bcplx$ since certain computations are simpler here.
\end{remark}

\begin{definition}
	Given a bigraded module $X$, the graded module $\tot(X)$ is defined as
	\[\tot_n(X)=\bigoplus_{p+q=n}X_{p,q}.\]
	If $X$ is a bicomplex, $\tot(X)$ equipped with the differential 
		\[d_{\tot}=d_h+d_v\]
		is called the \emph{total complex} of $X$. 
		This construction defines the \emph{totalisation} exact functor
		\[\tot\colon \bcplx\To\cplx.\]
\end{definition}

\begin{remark}
The totalisation functor is strong symmetric monoidal in the obvious naive way. In addition, $\tot$ preserves (co)limits, since they are computed pointwise both in $\bcplx$ and $\cplx$. If we had used bicomplexes with commuting differentials, we would have had to include signs in the natural isomorphism comparing the tensor products in the source and in the target of $\tot$. 
\end{remark}

\begin{remark}\label{spectral_sequence}
	For any bigraded  module $X$, $\tot(X)$ has a natural increasing non-negative exhaustive filtration defined by
	\[F_m\tot_n(X)=\bigoplus_{\substack{p+q=n\\p\leq m}}X_{p,q}.\]
	If $X$ is a bicomplex, this filtration of the total complex $\tot(X)$ is compatible with the differential. Since our bicomplexes are concentrated in the right half-plane, the associated spectral sequence converges strongly to the homology of $\tot(X)$. The $E^2$-term is $H^h(H^v(X))$ \cite[Theorem 2.15]{McCleary},
	\[E^2_{p,q}=H^h_{p,q}(H^v(X))\Longrightarrow H_{p+q}(\tot(X)).\]
	This will play a central role in the model structures to be defined later.
\end{remark}

\section{The total model structure on bicomplexes}\label{sec:total}

Let $\mathcal{W}$ denote the class of $(H_*\circ \tot)$-isomorphisms in $\bcplx$, i.e.~maps which induce a quasi-isomorphism in totalisation, and let 
\begin{align*}
		I_{\tot}&=\{S^{0,q-1}\mono\partial_hD^{1,q}\}_{q\in\mathbb Z}\cup\{\partial_vD^{p,q}\mono D^{p,q}\}_{\substack{p>0\\q\in\mathbb Z}},\\
		J_{\tot}&=\{0\mono \partial_hD^{1,q}\}_{q\in\mathbb Z}\cup\{\partial_vD^{p,q}\mono D^{p,q}\}_{\substack{p>0\\q\in\mathbb Z}}.
	\end{align*}
In the same way that one constructs the projective model structure on chain complexes outlined in Section \ref{sec:complexes}, we are going to use Smith's recognition principle to show that this choice defines a cofibrantly generated model structure on $\bcplx$. We will then further characterise its cofibrations and fibrations.

\begin{theorem}\label{total_bicomplexes}
	The category of bicomplexes $\bcplx$ can be endowed with a proper combinatorial abelian model category structure called the \emph{total model structure} with the following properties:
\begin{itemize}
\item a morphism $f\colon X\to Y$ is a weak equivalence if $\tot(f)$ is a quasi-isomorphism in $\cplx$, i.e.~the class of weak equivalences is $\mathcal{W}$,
\item a morphism $f\colon X\to Y$ is a (trivial) fibration if it is pointwise surjective and $H^v_{p,\ast}(f)$ is an isomorphism for all $p>0$ (resp.~$p\geq 0$), % and $q\in\mathbb{Z}$,
\item the cofibrations are the injective maps with cofibrant cokernel. Cofibrant implies pointwise projective.
\end{itemize}
Furthermore, its generating cofibrations and trivial cofibrations are given by $I_{\tot}$ and $J_{\tot}$, respectively.
\end{theorem}

\begin{proof}
	The proposed weak equivalences in $\bcplx_{\tot}$ are clearly closed under retracts and satisfy the $2$-out-of-$3$ property. 
	
\medskip
We will now verify the various lifting properties we require for our proof, making use of the identities in Lemma \ref{representation} and Corollary \ref{cor:adjunctions}. We use them in combination with the model structures on $\cplx$ and $\cplx_{\geq 0}$ reviewed in the previous sections, whose generating (trivial) cofibrations we know.

\medskip
A map $f\colon X \to Y$ in $\bcplx$ has the right lifting property with respect to $0 \rightarrow \partial_h D^{1,q}$ if and only if $f_{0,\ast}\colon X_{0,\ast} \rightarrow 	Y_{0, \ast}$ has the right lifting property with respect to $0\rightarrow D^q$. This happens for all $q\in\mathbb Z$ whenever $f_{0,\ast}$ is a fibration in $\cplx$, i.e.~degreewise surjective.

For $p>0$, having the right lifting property with respect to $\partial_vD^{p,q} \rightarrow D^{p,q}$ is equivalent to $f_{p,\ast}: X_{p,\ast} \rightarrow Y_{p, \ast}$ having the right lifting property with respect to all $S^{q-1} \rightarrow D^q$. This happens for all $q\in\mathbb Z$ precisely when $f_{p,\ast}$ is a trivial fibration in $\cplx$, i.e.~degreewise surjective as well as a homology isomorphism.

Similarily, having the right lifting property with respect to $S^{0,q-1} \rightarrow \partial_hD^{1,q}$ is equivalent to $f_{0,\ast}: X_{0, \ast} \rightarrow Y_{0,\ast}$ having the right lifting property with respect to $S^{q-1} \rightarrow D^q$ in $\cplx$. This happens for all $q\in\mathbb Z$ whenever $f_{0,\ast}$ is a trivial fibration in $\cplx$. 

We can summarise our findings as follows.
\begin{align*}
I_{\tot}\mbox{-inj}=\{\mbox{degreewise}&\mbox{ surjective $f$ such that}\\&\mbox{$H^v_{p,*}(f)$ is an isomorphism for all $p\ge0$} \}
\end{align*}
and
\begin{align*}
J_{\tot}\mbox{-inj}=\{\mbox{degreewise}&\mbox{ surjective $f$ such that}\\&\mbox{$H^v_{p,*}(f)$ is an isomorphism for all $p>0$}. \}
\end{align*}
This is consistent with our claims about the (trivial) fibrations in $\bcplx_{\tot}$. 

\bigskip
Obviously, $I_{\tot}\mbox{-inj}\subseteq J_{\tot}\mbox{-inj}$. 
 Because of the spectral sequence in Remark \ref{spectral_sequence}, a map that induces an isomorphism in vertical homology also induces an $(H_*\circ\tot)$-isomorphism. Therefore, a map in $I_{\tot}$-inj is also in $\mathcal{W}$ so
\[
I_{\tot}\mbox{-inj} \subseteq \mathcal{W} \cap J_{\tot}\mbox{-inj}.
\]
We also have the opposite inclusion. 

By the long exact homology sequence, a degreewise surjective map $f$ is a weak equivalence, respectively an $H^v_{p,\ast}$-isomorphism, if and only if $\ker f \rightarrow 0$ is. But if $X\to 0$ is $J_{\tot}$-injective, then $H^v_{0,*}(X)=H_*(\tot(X))$ by the spectral sequence. Therefore, a $J_{\tot}$-injective $X\rightarrow 0$ is $I_{\tot}$-injective  if and only if it is a weak equivalence.

So, altogether we arrive at
\[
I\mbox{-inj} = \mathcal{W} \cap J\mbox{-inj}.
\]
	
\bigskip
For the existence of the claimed model structure, it remains to prove that
\[
J_{\tot}\mbox{-cell} \subseteq \mathcal{W} \cap I_{\tot}\mbox{-cof}.
\]	
We always have $J_{\tot}\mbox{-cell} \subseteq J_{\tot}\mbox{-cof}.$ As $I_{\tot}\mbox{-inj} \subseteq J_{\tot}\mbox{-inj}$, we get $J_{\tot}\mbox{-cof} \subseteq I_{\tot}\mbox{-cof}.$ All source and target objects of elements of $J_{\tot}$ are $\tot$-acyclic, therefore $J_{\tot} \subseteq \mathcal{W}$. Moreover, $\tot$ preserves colimits and takes $J_{\tot}$ to trivial cofibrations in $\cplx$. Thus, $J_{\tot}$-cell complexes are weak equivalences, which is what we wanted to prove. Thus, we proved all the conditions of the recognition principle, meaning that we have a model structure on $\bcplx$ with weak equivalences $\mathcal{W}$, generating cofibrations $I_{\tot}$, generating acyclic cofibrations $J_{\tot}$, fibrations $J_{\tot}$-inj, and trivial fibrations $I_{\tot}$-inj.

\bigskip
The $I_{\tot}$-cell complexes are injections with pointwise free cokernel, since a push-out along one of the two classes of maps in $I_{\tot}$ adds a free factor $k$ in bidegrees  $(0,q)$ or $(p,q)$ and $(p-1,q)$, respectively. Hence cofibrations are injections with pointwise projective cokernel. This can be alternatively checked as in the first parts of the proofs of \cite[Proposition 2.3.9 and Lemma 2.3.6]{hovey_model_1999}. A long exact sequence argument as above shows that a map $f\colon X\rightarrow Y$ in $\bcplx_{\tot}$ is a (trivial) fibration if and only if it is a pointwise surjection with (trivially) fibrant kernel. These remarks prove that $\bcplx_{\tot}$ is an abelian model category in the sense of \cite[Definition 2.1]{hovey_cotorsion_2007}, see \cite[Proposition 4.2]{hovey_cotorsion_2002}, and also the third item in the statement.

\bigskip	
Finally, the functor $\tot$ takes $I_{\tot}$ to cofibrations in $\cplx$, and it also preserves (co)limits, fibrations, and weak equivalences (the latter by definition). We conclude that $\bcplx_{\tot}$ is proper, since $\cplx$ is.
\end{proof}

\begin{remark}\label{total_model_remarks}
	By the characterization of (trivial) fibrations in the total model structure, a bicomplex $X$ is fibrant in $\bcplx_{\tot}$ whenever its vertical homology is concentrated in horizontal degree $0$, i.e.~$H^v_{p,q}(X)=0$ if $p>0$. It is trivially fibrant if the vertical homology vanishes completely.
	
	The cokernels of generating cofibrations are $S^{0,q}$, $\partial_vD^{p,q}$, $p>0$, $q\in\mathbb Z$. The horizontal homology of these bicomplexes is projective and concentrated in horizontal degree $0$. Hence it is easy to derive that any cofibrant bicomplex $X$ satisfies $H_{p,q}^h(X)=0$ for $p>0$ and $H_{0,q}^h(X)$ is always projective. 
	
\end{remark}

The model structure on $\bcplx_{\tot}$ is also well-behaved with regards to the tensor product of bicomplexes. In order to show monoidality, it does not matter that we have not given an explicit characterisation of the cofibrations in $\bcplx_{\tot}$ as we can use a result specific to abelian model categories.

\begin{proposition}\label{prop:monoidal}
The model category $\bcplx_{\tot}$ is monoidal with cofibrant unit and satisfies the monoid axiom.
\end{proposition}

\begin{proof}
For the monoidality of $\bcplx_{\tot}$, we check the conditions of \cite[Theorem 4.2]{hovey_cotorsion_2007}:
\begin{enumerate}
\item every cofibrant object of $\bcplx_{\tot}$ is flat,
\item the tensor product of cofibrant objects is again cofibrant,
\item if $X$ and $Y$ are cofibrant objects and one of them is acyclic, then their tensor product is acyclic,
\item the unit of the tensor product is cofibrant.
\end{enumerate}

Cofibrant objects are flat since their underlying bigraded modules are projective. The tensor unit is cofibrant since it is the cokernel of the generating cofibration $S^{0,-1}\mono \partial_hD^{1,0}$. Therefore we have (1) and (4).

\bigskip
Let us now check conditions (2) and (3). For a general model category, it is enough to check the pushout-product axiom on generating (trivial) cofibrations, see \cite[Corollary 4.2.5]{hovey_model_1999}. So for abelian model categories, it is sufficient to check (2) and (3) on 
 cokernels of generating (trivial) cofibrations, rather than on arbitrary (trivially) cofibrant objects. This follows from the proof of \cite[Theorem 4.2]{hovey_cotorsion_2007} in \cite[Theorem 7.2]{hovey_cotorsion_2002}. 

\bigskip
 The cokernels of generating (trivial) cofibrations are 
 \[
 S^{0,q} \,\,\,\mbox{and} \,\,\, \partial_vD^{p,q}, \,\,\, p > 0, \,\, q \in \mathbb{Z},
 \]
as well as
 \[
 \partial_h D^{1,q} \,\,\,\mbox{and}\,\,\, \partial_vD^{p,q}, \,\,\, p > 0, \,\, q \in \mathbb{Z},
 \]
 respectively.
It is straightforward to verify that 
\begin{eqnarray*}
\partial_vD^{p,q} \otimes \partial_vD^{s,t}  & \cong & \partial_vD^{p+s,q+t-1}\oplus \partial_vD^{p+s-1,q+t-1}, \\
\partial_vD^{p,q}\otimes S^{0,t} & \cong & \partial_vD^{p,q+t}, \\
S^{0,q}\otimes S^{0,t} & \cong  & S^{0,q+t}.
\end{eqnarray*}
Moreover, 
\begin{eqnarray*}
	S^{0,q}\otimes \partial_hD^{1,t} & \cong  & \partial_hD^{1,t+q},\\
	\partial_vD^{p,q} \otimes \partial_hD^{1,t}  & \cong & D^{p,q+t-1}, 
\end{eqnarray*}
which we can see is trivially cofibrant for all $p$ and $q$. This concludes the proof of monoidality.

\bigskip
The monoid axiom \cite[Definition 3.3]{schwede_algebras_2000} follows from the fact that $\tot$ takes generating trivial cofibrations in $\bcplx_{\tot}$ to trivial cofibrations in $\cplx$.
\end{proof}

Finally, we arrive at the following result. 

\begin{proposition}\label{quillen_equivalence_bicomplexes_complexes}
	The inclusion of chain complexes as bicomplexes concentrated in horizontal degree $0$ is the left adjoint of a strong symmetric monoidal Quillen equivalence
	\[\cplx \lradjunction \bcplx_{\tot}. \]
\end{proposition}

\begin{proof}

	The right adjoint $\bcplx_{\tot} \rightarrow \cplx$ in this adjunction is given by  $X\mapsto X_{0,*}$. The left adjoint obviously preserves the tensor product and the tensor unit. Moreover, it sends the standard generating (trivial) cofibrations in $\cplx$ to 
	the first factors of the unions defining the sets of generating (trivial) cofibrations of $\bcplx_{\tot}$. Hence it preserves arbitrary (trivial) cofibrations and therefore is a left Quillen functor.

	\bigskip
	Next, let us check that this Quillen pair is a Quillen equivalence. Let $$i\colon \cplx\rightarrow\bcplx$$ denote the left adjoint. We have to prove that if $X$ is cofibrant in $\cplx$ and $Y$ is fibrant in $\bcplx_{\tot}$ then a map $$f\colon iX\rightarrow Y$$ is a weak equivalence if and only if the adjoint map $$f_{0,*}\colon X\rightarrow Y_{0,*}$$ is. We will prove this statement for $X$ not necessarily cofibrant. We have that $$\tot(iX)=X=(iX)_{0,*}.$$ Moreover, since $Y$ is fibrant, we have  $$H_{*}(Y_{0,*})=H^v_{0,*}(Y)=H_v(\tot(Y)).$$ This was checked within the proof of Theorem \ref{total_bicomplexes}. Hence $$H_*(\tot(iX))\rightarrow H_*(\tot(Y))$$ is an isomorphism if and only if $H_*(X)\rightarrow H_*(Y_{0,*})$ is.
\end{proof}

\section{The Cartan--Eilenberg model structure on bicomplexes}\label{sec:ce}

In this section, we will introduce a different model structure on the category of bicomplexes $\bcplx$, namely the \emph{Cartan--Eilenberg model structure} $\bcplx_{\CE}$. With regards to the spectral sequence in Remark \ref{spectral_sequence}, the total model structure from Section \ref{sec:total} can be thought of as the ``limit model structure'' as the weak equivalences are exactly those maps inducing isomorphisms on the respective limits. (By limit, we mean the homology of the total complex, not the $E^\infty$-term.) In analogy to this, the Cartan--Eilenberg model structure can be considered the ``$E^2$-model structure''. The spectral sequence is strongly convergent, hence the weak equivalences of $\bcplx_{\CE}$ are contained in those of $\bcplx_{\tot}$. In this new model category $\bcplx_{\CE}$, a cofibrant resolution of a chain complex, regarded as bicomplex concentrated in horizontal degree $0$, is a Cartan--Eilenberg resolution \cite{cartan_homological_1956, Weibel_1994}.

	We suspect that the Cartan--Eilenberg model structure coincides with the model structure that could be obtained by transferring Sagave's $E^2$-model structure on simplicial chain complexes \cite{sagave_dg-algebras_2010} along the given Dold--Kan equivalence. Weak equivalences obviously match, but we have not checked the details concerning (co)fibrations. We think it is simpler to directly construct our Cartan--Eilenberg model structure, at the very least because we obtain easier generating (trivial) cofibrations which allow for a straightforward identification of (trivial) fibrations.
	For $0\leq r<\infty$, $E^r$-equivalences have been studied by Cirici, Santander, Livernet, and Whitehouse in \cite{cirici-egas-livernet-whitehouse}, not only for maps of bicomplexes but for twisted maps of twisted complexes.

\bigskip
Again, we will define weak equivalences, generating cofibrations, and trivial cofibrations, and then show that they create a model structure in the way that they are supposed to. Define
\[
\mathcal{W} = \{ f: X \rightarrow Y \in \bcplx \,\,|\,\, H^h_{p,q}(H^v(f)) \mbox{ is an isomorphism for all $p\ge 0, q \in \mathbb{Z}$} \}
\]
and 
\begin{align*}
	I_{\CE}={}&
	\{0\mono S^{0,q}\}_{q\in\mathbb Z}\cup
	\{S^{p-1,q-1}\mono\partial_vD^{p,q}\}_{\substack{p>0\\q\in\mathbb Z}}\\&\cup\{0\mono\partial_hD^{1,q}\}_{q\in\mathbb Z}\cup\{\partial_hD^{p,q}\mono D^{p,q}\}_{\substack{p>0\\q\in\mathbb Z}},\\
	J_{\CE}={}&\{0\mono\partial_vD^{p,q}\}_{\substack{p>0\\q\in\mathbb Z}}\cup\{0\mono\partial_hD^{1,q}\}_{q\in\mathbb Z}\cup\{\partial_hD^{p,q}\mono D^{p,q}\}_{\substack{p>0\\q\in\mathbb Z}}.
	\end{align*}

\begin{theorem}\label{Cartan_Eilenberg}
	There is a combinatorial model structure $\bcplx_{\CE}$ on the category of bicomplexes satisfying the following:
	\begin{itemize}
	\item Weak equivalences are those morphisms $f$ such that $H^h_{p,q}(H^v(f))$ is an isomorphism for all $p\geq 0$ and $q\in\mathbb Z$.
	\item $f$ is a fibration if
	\begin{itemize}
	\item if is degreewise surjective,
	\item $Z^v_{p,q}(f)$ is surjective if $p>0$,
	\item  $H^h_{p,q}(f)$ is an isomorphism for all $p$ and $q$.
	\end{itemize}
	\item $f$ is a trivial fibration if in addition $H^h_{p,q}(Z^v(f))$ is an isomorphism for all $p$ and $q$.
	\item If $f$ is a cofibration, then $f$ and $H^v(f)$ are injective and the cokernels of $B^v(f)$ and $H^v(f)$ are degreewise projective.
	\end{itemize}
	
\end{theorem}

\begin{proof}
We are going to follow a similar strategy to the proof of Theorem \ref{total_bicomplexes}.
	Weak equivalences in $\bcplx_{\CE}$ are obviously closed under retracts and transfinite compositions, and satisfy the $2$-out-of-$3$ property. 
	
	\bigskip
	Using Corollary \ref{cor:adjunctions} and the generating (trivial) cofibrations of $\cplx_{\geq 0}$, we see the following. A morphism $f: X \rightarrow Y$ in $\bcplx$ has the right lifting property with respect to all elements in $I_{\CE}$ if and only if $$Z^v_{*,q}(f): Z^v_{*,q}(X) \longrightarrow Z^v_{*,q}(Y),\qquad f_{*,q}: X_{*,q} \longrightarrow Y_{*,q},$$ are trivial fibrations in $\cplx_{\ge 0}$ for all $q$. Similarily, $f$ has the right lifting property with respect to all elements in $J_{\CE}$ if and only if $Z^v_{*,q}(f)$ is a fibration and $f_{*,q}$ is a trivial fibration, both in $\cplx_{\ge 0}$, for all $q$. Therefore, we obviously have $$I_{\CE}\mbox{-inj}\subseteq J_{\CE}\mbox{-inj}.$$ 
	
	We also have $I_{\CE}\mbox{-inj}\subseteq \mathcal{W}$. This follows by considering the obvious natural short exact sequences of complexes on $\cplx_{\geq 0}$, $q\in\mathbb Z$,
	\[Z^v_{*,q}(X)\mono X_{*,q}\onto B^v_{*,q-1}(X),\qquad
	B^v_{*,q}(X)\mono Z^v_{*,q}(X)\onto H^v_{*,q}(X).\]
	They show that if a map of bicomplexes $f\colon X\rightarrow Y$ induces quasi-isomorphisms on horizontal complexes and vertical cycles \[f_{*,q}\colon X_{*,q}\rightarrow Y_{*,q},\qquad Z^v_{*,q}(f)\colon Z^v_{*,q}(X)\rightarrow Z^v_{*,q}(Y),\] for all $q\in\mathbb Z$, then it also induces quasi-isomorphisms on vertical boundaries and vertical homology  \[B^v_{*,q}(f)\colon B^v_{*,q}(X)\rightarrow B^v_{*,q}(Y),\qquad H^v_{*,q}(f)\colon H^v_{*,q}(X)\rightarrow H^v_{*,q}(Y).\] Altogether,
	\[
	I_{\CE}\mbox{-inj}\subseteq J_{\CE}\mbox{-inj} \cap \mathcal{W}
	\]
	as required.

\bigskip

	It is not hard to see that any map in $J_{\CE}$, and hence any $J_{\CE}$-cell complex, is an $I_{\CE}$-cell complex, so it is in $I_{\CE}$-cof. Indeed, the two last factors of the union $J_{\CE}$ are also in $I_{\CE}$. Moreover, the maps $$0\mono S^{p-1,q-1}, \,\,\,p>0,\,\,\,q\in\mathbb Z,$$ are $I_{\CE}$-cell complexes since $S^{p-1,q-1}$ is the cokernel of the map $$S^{p-2,q-1}\mono\partial_vD^{p-1,q},$$ in $I_{\CE}$. Hence $0\mono\partial_vD^{p,q}$, which is the composite of $$0\mono S^{p-1,q-1} \,\,\,\mbox{and}\,\,\,S^{p-1,q-1}\mono\partial_vD^{p,q},$$ is also an $I_{\CE}$-cell complex.

	 We can also see that every $J_{\CE}$-cell complex is a weak equivalence: A push-out along one of the first two kinds of maps in $J_{\CE}$ adds up a copy of $\partial_vD^{p,q}$ or $\partial_hD^{1,q}$, which are weakly equivalent to $0$. Hence such a push-out is a weak equivalence. A push-out along one of the third kind adds a copy of $k$ to the modules of bidegrees $(p,q)$ and $(p,q-1)$, and $d_v$ maps identically the top copy to the bottom one. Hence it induces an isomorphism on $H^v$, in particular it is a weak equivalence. Since any transfinite composition of weak equivalences in this tentative model structure is a weak equivalence, we derive that any $J_{\CE}$-cell complex is a weak equivalence. Therefore,
	\[
	J_{\CE}\mbox{-cell} \subseteq \mathcal{W} \cap I_{\CE}\mbox{-cof}.
	\]

	\bigskip
	To complete the proof of this model structure's existence, we have to show that
	\[
		 J_{\CE}\mbox{-inj} \cap \mathcal{W} \subseteq I_{\CE}\mbox{-inj}.
	\]

	Assume that $f\colon X\to Y$ is a $J_{\CE}$-injective weak equivalence. 
	We want to show that
	\[
	H^h_{p,q}(f) \,\,\,\mbox{is an isomorphism for all $p$ and $q$}
	\]
and
\[
H^h_{p,q}(H^v(f)) \,\,\,\mbox{is an isomorphism for all $p$ and $q$}
\]	
implies that 
\[
H^h_{p,q}(Z^v(f)) \,\,\,\mbox{is an isomorphism for all $p$ and $q$}.
\]

\bigskip
	 	
	We prove by induction on $p$ that, for all $q\in\mathbb Z$, $H_{n,q}^h(Z^v(f))$ and $H_{n,q}^h(B^v(f))$ are isomorphisms for $n<p$ and epimorphisms for $n=p$. This will suffice. For this, we will use the long exact homology sequences associated to the two natural short exact sequences of horizontal complexes above.
	
	\bigskip
	The statement is obvious for $p=-1$ since our bicomplexes are trivial in negative horizontal degrees. 
	So next, assume the result true for $p$. 	
	The exact sequences of maps
	\[\underbrace{H^{h}_{p,q}(Z^v(f))}_{\text{epi}}\to \underbrace{H^{h}_{p,q}(f)}_{\text{iso}}\to H^{h}_{p,q-1}(B^v(f))\to \underbrace{H^{h}_{p-1,q}(Z^v(f))}_{\text{iso}}\to \underbrace{H^{h}_{p-1,q}(f)}_{\text{iso}}\]
	for $q\in\mathbb Z$ and the Five Lemma show that $H^{h}_{p,q}(B^v(f))$ is an isomorphism for all $q$. Now, the exact sequences
	\[\underbrace{H^{h}_{p+1,q}(H^v(f))}_{\text{iso}}\to \underbrace{H^{h}_{p,q}(B^v(f))}_{\text{iso}}\to H^{h}_{p,q}(Z^v(f))\to \underbrace{H^{h}_{p,q}(H^v(f))}_{\text{iso}} \to \underbrace{H^{h}_{p-1,q}(B^v(f))}_{\text{iso}}\]
	prove that $H^{h}_{p,q}(Z^v(f))$ is an isomorphism for $q\in\mathbb Z$ for the same reasons. Chasing
	\[\underbrace{H^{h}_{p+1,q}(f)}_{\text{iso}}\to H^{h}_{p+1,q-1} (B^v(f))\to 
		\underbrace{H^{h}_{p,q}(Z^v(f))}_{\text{iso}}\to \underbrace{H^{h}_{p,q}(f)}_{\text{iso}}\]
	we see that $H^{h}_{p+1,q} (B^v(f))$ is an epimorphism for all $q$, and chasing
	\[\underbrace{H^{h}_{p+1,q}(B^v(f))}_{\text{epi}}\to H^{h}_{p+1,q} (Z^v(f))\to 
	\underbrace{H^{h}_{p+1,q}(H^v(f))}_{\text{iso}}\to \underbrace{H^{h}_{p,q}(B^v(f))}_{\text{iso}}\]
	we deduce that $H^{h}_{p+1,q} (Z^v(f))$ is also an epimorphism for all $q$.
	
	So again we have checked the conditions for \cite[Theorem 2.1.19]{hovey_model_1999}, hence the model structure $\bcplx_{\CE}$ with the indicated weak equivalences and (trivial) fibrations exists.
	
	\bigskip
	Now let us take a look at the cofibrations.
		There are four kinds of generating cofibrations. A push-out along one of the first two kinds adds up a copy of $k$ of bidegree $(p,q-1)$, $p\geq 0$ with trivial vertical differential. As we remarked earlier, a push-out along one of the last two kinds adds a copy of $k$ of bidegrees $(p,q)$ and $(p,q-1)$, $p\geq 0$, and $d_v$ maps identically the top copy to the bottom one. It follows by induction that, for any $I_{\CE}$-cofibration $f$, both $f$ and $H^v(f)$ are injective and the cokernels of $B^v(f)$ and $H^v(f)$ are pointwise projective. 
		
	\end{proof}

		Given a chain complex $Y$ regarded as a bicomplex concentrated in the horizontal degree $0$, any cofibrant resolution $ Y^{cof} \onto Y$ in the Cartan--Eilenberg model structure $\bcplx_{\CE}$ is a projective resolution in the sense of \cite[\S XVII.1]{cartan_homological_1956}, hence the name of the model structure. It is also what Sagave calls a ``$k$-projective $E_1$-resolution'' in  \cite{sagave_dg-algebras_2010}.
	
	\bigskip
	A bicomplex $X$ is fibrant in $\bcplx_{\CE}$ if and only if its horizontal homology is trivial $H^h(X)=0$. It is trivially fibrant if in addition $H^h(Z^v(X))=0$.
	
	\bigskip
	It is possible to check with a certain amount of work that the model category $\bcplx_{\CE}$ is proper. It is not abelian since, for $p>0$, the projection $D^{p,q}\onto \partial_vD^{p,q+1}$ onto the cokernel of $\partial_vD^{p,q}\mono D^{p,q}$ has a fibrant kernel $\partial_vD^{p,q}$, but it is not a fibration as it is not surjective on $Z^v_{p,q}$.
	Nevertheless, it could be compatible with a restricted family of short exact sequences in the sense of \cite{hovey_cotorsion_2002}.
	
	\bigskip
Again, this model structure is well-behaved with regards to the tensor product.	

\begin{proposition}
The model category $\bcplx_{\CE}$ is monoidal with cofibrant tensor unit and satisfies the monoid axiom. 
\end{proposition}

\begin{proof}

	For $q\in\mathbb Z$, we have functors $$z_q,c_q\colon \cplx_{\geq 0}\to\bcplx$$ defined as follows. Given an object $X$ in the source, the bicomplex $z_q(X)$ is $X$ concentrated in vertical degree $q$ and zero elsewhere. The complex $c_q(X)$ is obtained by placing $X$ in vertical degrees $q$ and $q-1$ and taking the vertical differential $d_v$ from bidegree $(p,q)$ to $(p,q-1)$ to be $(-1)^p$. (The sign is needed to get anticommuting differentials.)
	
	\bigskip
	We have
\[
	\begin{array}{rl}
	z_{q-1}(D^p)=\partial_vD^{p,q}, & z_q (S^p)=S^{p,q}, \\
	c_q(D^p)=D^{p,q}, & c_q(S^{p-1})=\partial_h D^{p,q}.
	\end{array}
	\]
	
	Therefore, $z_q$ sends the elements of $I_{\cplx_{\ge 0}}$ to elements of $I_{\CE}$ and $J_{\cplx_{\ge 0}}$ to $J_{\CE}$. This implies that the $z_q$ preserve cofibrations and trivial cofibrations.
	Likewise, the $c_q$ send cofibrations in $\cplx_{\ge 0}$ to trivial cofibrations in $\bcplx_{\CE}$.

	\bigskip
	The model category $\cplx_{\geq 0}$ is monoidal, and we have natural isomorphisms $$z_p(X)\otimes z_q(Y)\cong z_{p+q}(X\otimes Y).$$ Thus, we see that the push-out product of two cofibrations concentrated a single (possibly different) vertical degree is a cofibration in $\bcplx_{\CE}$, which is trivial if one of the initial cofibrations was. 
	
	Moreover, if $f$ is a cofibration in $\cplx_{\geq0}$ and $U$ is a trivially cofibrant object in $\cplx$ regarded as a bicomplex concentrated in horizontal degree $0$, then $U\otimes z_q(f)$ is a trivial cofibration in $\bcplx_{\CE}$. Indeed, the complex $U$, being trivially cofibrant, is a retract of a direct sum of copies of $D^t$, $t\in\mathbb Z$, i.e.~$U$ regarded as a bicomplex is a retract of a direct sum of copies of $\partial_hD^{1,t}$, $t\in\mathbb Z$. Moreover, $$\partial_hD^{1,t}\otimes z_q(f)\cong c_{t+q}(f),$$ hence $U\otimes z_q(f)$ is retract of a direct sum factors of the form $c_t(f)$, $t\in\mathbb Z$, which are trivial cofibrations. Recall also that trivially cofibrant objects in $\cplx$ are closed under tensor products since the model category $\cplx$ is monoidal.
	
	If we combine the previous observations with the fact that the map $\partial_hD^{p,q}\mono D^{p,q}$ is the same as $$\partial_hD^{1,1}\otimes(S^{p-1,q-1}\mono\partial_vD^{p,q}),$$ we conclude that the push-out product of two maps in $I_{\CE}$ is a cofibration, and that the push-out product of a map in $I_{\CE}$ and a map in $J_{\CE}$ is a trivial cofibration. Hence $\bcplx_{\CE}$ is monoidal.

\bigskip

	Let us now consider the monoid axiom \cite[Definition 3.3]{schwede_algebras_2000}. Since Cartan--Eilenberg equivalences are closed under transfinite compositions, it suffices to prove that, given two bicomplexes $X$ and $Y$, the push-out of $X$ along a map $f$ in $Y\otimes J_{\CE}$ is a weak equivalence. All maps in $J_{\CE}$ are pointwise split monomorphisms, hence maps in $Y\otimes J_{\CE}$ are injective and therefore the push-out along a map $f$ in $Y\otimes J_{\CE}$ is a weak equivalence if and only if the cokernel of $f$ is acyclic. The cokernels of maps in $Y\otimes J_{\CE}$ are of the form $Y\otimes\partial_vD^{p,q}$ and $Y\otimes\partial_hD^{p,q}$, $p> 0$, $q\in\mathbb Z$. They are acyclic in $\bcplx_{\CE}$ because
	\begin{align*}
		H^h(H^v(Y\otimes\partial_vD^{p,q}))&=H^h(H^v(Y)\otimes\partial_vD^{p,q})=0,\\
		H^v(Y\otimes\partial_hD^{p,q})&=0.
	\end{align*}
	This concludes the verification of the monoid axiom. 
	
	The tensor unit $S^{0,0}$ is cofibrant since $0\mono S^{0,0}$ is one of the generating cofibrations.

\end{proof}

\section{Twisted complexes and their total model structure}\label{sec:twisted}

In this section, we generalise the total model structure on $\bcplx$ to the category of twisted complexes. The notion of twisted complex goes back to Wall \cite{Wall_1961}. They have proven useful in many contexts in the construction of small resolutions.

\begin{definition}\label{twisted_complexes}
	A \emph{twisted complex} $X$, also known as \emph{multicomplex}, is a bigraded module $X=\{X_{p,q}\}_{p,q\in\mathbb Z}$ with $X_{p,q}=0$ for $p<0$ equipped with maps 
	\[d_i\colon X_{p,q}\To X_{p-i,q+i-1},\quad i\geq 0,\]
	satisfying
	\[\sum_{i+j=n} d_id_j=0,\quad n\geq 0.\]
	Abusing terminology, we also call the maps $d_i$ differentials, despite they do not square to zero in general.
	
	
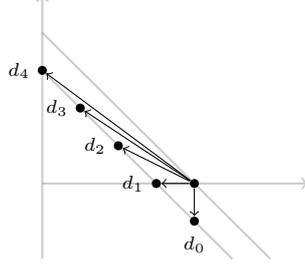
\begin{center}
		\begin{tikzpicture}[execute at begin node=$\scriptstyle , execute at end node=$]
		\tikzstyle{every node}=[circle, fill, inner sep=1.2]
		\draw[->, thick, opacity =.2] (0,0) -- (3.5,0);
		\draw[->, thick, opacity =.2] (0,-1) -- (0,2.5);
		\draw[thick, opacity =.2] (0,2) -- (3,-1);
		\draw[thick, opacity =.2] (0,1.5) -- (2.5,-1);
		\node at (2,0) (A) {};
		\node[label=below:{d_0}] at (2,-.5) (B) {};
		\node[label=left:{d_1}] at (1.5,0) (C) {};
		\node[label=left:{d_2}] at (1,.5) (D) {};
		\node[label=left:{d_3}] at (.5,1) (E) {};
		\node[label=left:{d_4}] at (0,1.5) (F) {};
		\draw[->] (A) -- (B);
		\draw[->] (A) -- (C);
		\draw[->] (A) -- (D);
		\draw[->] (A) -- (E);
		\draw[->] (A) -- (F);
		\end{tikzpicture}
	\end{center}
	
	 A morphism of bicomplexes $f\colon X\to Y$ is a family of maps $f_{p,q}\colon X_{p,q}\to Y_{p,q}$ compatible with the differentials. We denote the category of twisted complexes by $\tcplx$. \end{definition}

\begin{definition}
We define the \emph{total complex} of the twisted complex $X$ as the chain complex which in degree $n$ is $$\tot(X)_n = \bigoplus\limits_{p+q=n} X_{p,q}.$$ 
The differential in $\tot(X)$ is then
	\[d=\sum_{i\geq 0}d_i.\]
	\emph{Totalisation} defines an functor
	\[\tot\colon \tcplx\To\cplx.\]
\end{definition}

	Note that the sum in the differential is finite on each bidegree $(p,q)$ since the target of $d_i$ is trivial for $i>p$. 
We see that the differential on $\tot(X)$ is compatible with the filtration by the horizontal degree considered in Remark \ref{spectral_sequence}. Moreover, any such differential in $\tot(X)$ comes from a unique twisted complex structure on $X$.

\bigskip
The category $\tcplx$ is clearly abelian and locally finitely presentable. Limits and colimits are computed pointwise. The totalisation functor is exact and preserves (co)limits.

\begin{remark}\label{twisted_properties}
	A bicomplex is the same as a twisted complex with $d_i=0$ for $i\geq 2$. The horizontal and vertical differentials are $d_h=d_1$ and $d_v=d_0$. This observation defines a fully faithful functor $\bcplx\to\tcplx$ which has a left adjoint $\tcplx\to\bcplx$ sending $X$ to
	\[X\Big\slash\Big(\sum_{i\geq 2}d_i(X)\Big).\]
	
	The first equations defining a general twisted complex $X$ are
	\begin{align*}
		d_0^2&=0,\\
		d_0d_1+d_1d_0&=0,\\
		d_0d_2+d_1^2+d_2d_0&=0.
	\end{align*}
	Hence $d_0$, which points vertically downwards, is always a differential, so we can define the \emph{vertical cycles} $Z^v(X)$ and \emph{vertical homology} $H^v(X)$. Moreover, $d_1$ induces a differential on $H^v(X)$ ponting horizontally to the left, hence we can define its \emph{horizontal homology} $H^h(H^v(X))$. As in Remark \ref{spectral_sequence}, this is the second term of the spectral sequence of the filtered complex $\tot(X)$, converging strongly to its homology,
	\[E^2_{p,q}=H^h_{p,q}(H^v(X))\Longrightarrow H_{p+q}(\tot(X)).\]
\end{remark}

\begin{definition}\label{twisted_disk}
	We define the \emph{twisted $(p,q)$-disc} $\tilde{D}^{p,q}$, $p\geq 0$, $q\in\mathbb Z$, as the twisted complex freely generated by a single element $x_{p,q}\in \tilde{D}^{p,q}_{p,q}$.
	More precisely, as a $k$-module $\tilde{D}^{p,q}_{p,q}=k$ generated by $x_{p,q}$, and for $0\leq s\leq p$ and $n\geq 1$, $\tilde{D}^{p,q}_{p-s,q+s-n}$ is the quotient of the free module generated by
	\[\{d_{i_1}\cdots d_{i_n}(x_{p,q})\}_{i_1+\cdots +i_n=s}\]
	by the relations
	\[
	\bigg\{\sum_{i_{j}+i_{j+1}=m}d_{i_1}\cdots d_{i_{j}}d_{i_{j+1}}\cdots d_{i_n}(x_{p,q})\bigg\}_{\substack{1\leq j<n\\i_1+\cdots+i_{j-1}+m+i_{j+2}+\cdots+i_n=s}}
	\]
	Elsewhere, $\tilde{D}^{p,q}$ is trivial. These relations provide a way of taking any $d_0$ in $d_{i_1}\cdots d_{i_n}(x_{p,q})$ to the right. In particular, $d_{i_1}\cdots d_{i_n}(x_{p,q})$ is trivial as long as there are two $0$ subscripts, since $d_0^2=0$. Therefore $\tilde{D}^{p,q}_{s,t}=0$ if $t<q-1$. Moreover, we can take any word $d_{i_1}\cdots d_{i_n}(x_{p,q})$ with exactly one $d_0$ in it and uniquely rewrite it as the sum of words without any $d_0$ and a word with exactly one $d_0$ at the right. More precisely, if the only trivial subscript is $i_j=0$ then
	\begin{multline*}
	d_{i_1}\cdots d_{i_{j-1}} d_0 d_{i_{j+1}} \cdots d_{i_n} (x_{p,q}) ={} (-1)^{n-j} d_{i_1}\cdots d_{i_{j-1}} d_{i_{j+1}} \cdots d_{i_n}d_0(x_{p,q})  \\ +
	\sum_{u=1}^{n-j}\sum\limits_{\substack{s+t=i_{j+u} \\ s, t > 0}} (-1)^ud_{i_1}\cdots d_{i_{j-1}} 
	\underbrace{d_{i_{j+1}}\cdots d_{i_{j+u-1}}}_{u-1\text{ factors}}
	d_{s} d_t  \underbrace{d_{i_{j+2}}d_{i_{j+u+1}}\cdots d_{i_n}(x_{p,q})}_{n-j-u\text{ factors}}.
	\end{multline*}
	One can straightforwardly check that the rewriting process consisting of replacing the word on the left with the sum on the right, removing words $d_{i_1}\cdots d_{i_n}(x_{p,q})$ with two $0$ subscripts, and not changing words with no $0$ subscript, sends the defining relations of $\tilde{D}^{p,q}$ to $0$.	Consequently, for $0\leq s\leq p$ and $1\leq n\leq s+1$, $\tilde{D}^{p,q}_{p-s,q+s-n}$ is freely generated by
	\[
	\{d_{i_1}\cdots d_{i_n}(x_{p,q})\}_{\substack{i_1,\dots,i_n>0\\i_1+\cdots +i_n=s}}
	\cup
	\{d_{i_1}\cdots d_{i_{n-1}}d_0(x_{p,q})\}_{\substack{i_1,\dots,i_{n-1}>0\\i_1+\cdots +i_{n-1}=s}}.
	\]
	Hence, the rank of this $\tilde{D}^{p,q}_{p-s,q+s-n}$ is $\binom{s-1}{n-1}+\binom{s-1}{n-2}$ if $1<n\leq s$, and $1$ if $n=1$ or $n=s+1$. The following picture gives a rough idea of how $\tilde{D}^{4,0}$ looks like
	
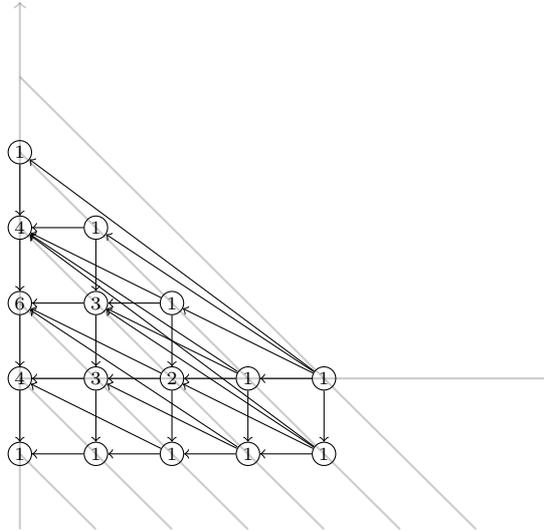
\begin{center}
		\begin{tikzpicture}[execute at begin node=$\scriptstyle , execute at end node=$, scale=2]
		\tikzstyle{every node}=[circle, draw, inner sep=1]
		\draw[->, thick, opacity =.2] (0,0) -- (3.5,0);
		\draw[->, thick, opacity =.2] (0,-1) -- (0,2.5);
		\draw[thick, opacity =.2] (0,2) -- (3,-1);
		\draw[thick, opacity =.2] (0,1.5) -- (2.5,-1);
		\draw[thick, opacity =.2] (0,1) -- (2,-1);
		\draw[thick, opacity =.2] (0,.5) -- (1.5,-1);
		\draw[thick, opacity =.2] (0,0) -- (1,-1);
		\draw[thick, opacity =.2] (0,-.5) -- (.5,-1);
		\node at (2,0) (A) {1};
		\node at (2,-.5) (B) {1};
		\node at (1.5,0) (C) {1};
		\node at (1,.5) (D) {1};
		\node at (.5,1) (E) {1};
		\node at (0,1.5) (F) {1};
		\draw[->] (A) -- (B);
		\draw[->] (A) -- (C);
		\draw[->] (A) -- (D);
		\draw[->] (A) -- (E);
		\draw[->] (A) -- (F);
		\node at (1.5,-.5) (G) {1};
		\node at (1,0) (H) {2};
		\node at (.5,.5) (I) {3};
		\node at (0,1) (J) {4};
		\draw[->] (B) -- (G);
		\draw[->] (B) -- (H);
		\draw[->] (B) -- (I);
		\draw[->] (B) -- (J);
		\draw[->] (C) -- (G);
		\draw[->] (C) -- (H);
		\draw[->] (C) -- (I);
		\draw[->] (C) -- (J);
		\draw[->] (D) -- (H);
		\draw[->] (D) -- (I);
		\draw[->] (D) -- (J);
		\draw[->] (E) -- (I);
		\draw[->] (E) -- (J);
		\draw[->] (F) -- (J);
		\node at (1,-.5) (K) {1};
		\node at (.5,0) (L) {3};
		\node at (0,.5) (M) {6};
		\draw[->] (G) -- (K);
		\draw[->] (G) -- (L);
		\draw[->] (G) -- (M);
		\draw[->] (H) -- (K);
		\draw[->] (H) -- (L);
		\draw[->] (H) -- (M);
		\draw[->] (I) -- (L);
		\draw[->] (I) -- (M);
		\draw[->] (J) -- (M);
		\node at (.5,-.5) (N) {1};
		\node at (0,0) (O) {4};
		\draw[->] (K) -- (N);
		\draw[->] (K) -- (O);
		\draw[->] (L) -- (N);
		\draw[->] (L) -- (O);
		\draw[->] (M) -- (O);
		\node at (0,-.5) (P) {1};
		\draw[->] (N) -- (P);
		\draw[->] (O) -- (P);
		\end{tikzpicture}
	\end{center}
	The nodes indicate the rank of the non-trivial parts of the underlying bigraded module. The non-trivial arrows are also depicted.
\end{definition}

\begin{remark}\label{basis_alternative}
One can define an analogous algorithm to move the $d_0$ in a word to the left. 
	Therefore,
		\[
		\{d_{i_1}\cdots d_{i_n}(x_{p,q})\}_{\substack{i_1,\dots,i_n>0\\i_1+\cdots +i_n=s}}
		\cup
		\{d_0d_{i_1}\cdots d_{i_{n-1}}(x_{p,q})\}_{\substack{i_1,\dots,i_{n-1}>0\\i_1+\cdots +i_{n-1}=s}}
		\]
	is an alternative basis of $\tilde{D}^{p,q}_{p-s,q+s-n}$ for $0\leq n-1\leq s\leq p$.
\end{remark}

For any $k$-module $A$, we can easily define $\tilde{D}^{p,q}(A)$ by degreewise tensoring with $A$, which gives an analogous construction with copies of $A$ instead of copies of $k$ at each node. By definition, this gives us an adjoint functor pair
\[
\tilde{D}^{p,q}: k\mbox{-mod} \lradjunction \tcplx: ev_{p,q},
\]
where $ev_{p,q}(X)=X_{p,q}$. 
In particular, for any twisted complex $X$, we have a natural isomorphism \[\tcplx(\tilde{D}^{p,q},X)=X_{p,q}.\]

\begin{lemma}\label{twisted_disk_trivial}
	Twisted disks have trivial total homology, $H_*(\tot(\tilde{D}^{p,q}))=0$.
\end{lemma}

\begin{proof}
	Using the bases in Remark \ref{basis_alternative}, the vertical homology $d_0$ applied to a basis element is either zero or of the form $d_0d_{i_1}\cdots d_{i_{n-1}}(x_{p,q})$. 
	Thus, their vertical homology vanishes, $H^v(\tilde{D}^{p,q})=0$, since $\tilde{D}^{p,q}$ is vertically contractible. Using the spectral sequence in Remark \ref{twisted_properties}, this implies that the homology of the total complex is also trivial. 
\end{proof}

\begin{definition}\label{twisted_boundary}
	We define the \emph{vertical boundary} $\partial_v\tilde{D}^{p,q}$ of the twisted $(p,q)$-disc, $p\geq 0$, $q\in\mathbb Z$, as the twisted complex freely generated by a single element $y_{p,q-1}\in \partial_v\tilde{D}^{p,q}_{p,q-1}$ satisfying $d_v(y_{p,q-1})=0$ (i.e.~a vertical cycle).
	
	Arguing as in Definition \ref{twisted_disk}, $\partial_v\tilde{D}^{p,q}_{p,q-1}=k$ generated by $y_{p,q-1}$, and for $0\leq s\leq p$ and $1\leq n\leq s$, $\partial_v\tilde{D}^{p,q}_{p-s,q-1+s-n}$ is freely generated by
	\[
	\{d_{i_1}\cdots d_{i_n}(y_{p,q-1})\}_{\substack{i_1,\dots,i_n>0\\i_1+\cdots +i_n=s}}.
	\]
	Hence its rank is $\binom{s-1}{n-1}$.
	Elsewhere, $\partial_v\tilde{D}^{p,q}$ is trivial. 	 
	
	We depict $\partial_v\tilde{D}^{4,0}$ with the same conventions as in Definition \ref{twisted_disk},
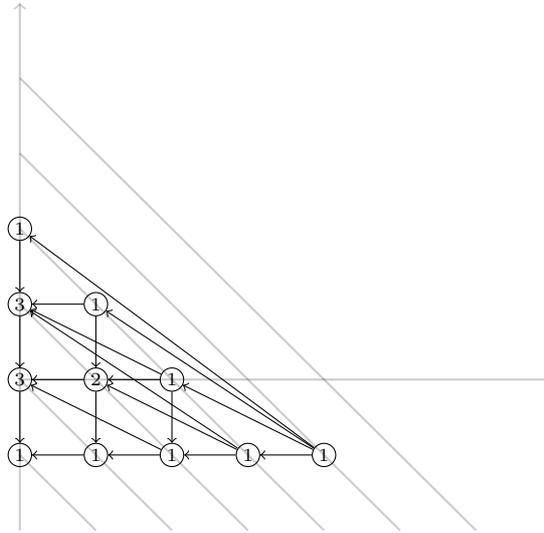
	\begin{center}
		\begin{tikzpicture}[execute at begin node=$\scriptstyle , execute at end node=$, scale=2]
		\tikzstyle{every node}=[circle, draw, inner sep=1]
		\draw[->, thick, opacity =.2] (0,0) -- (3.5,0);
		\draw[->, thick, opacity =.2] (0,-1) -- (0,2.5);
		\draw[thick, opacity =.2] (0,2) -- (3,-1);
		\draw[thick, opacity =.2] (0,1.5) -- (2.5,-1);
		\draw[thick, opacity =.2] (0,1) -- (2,-1);
		\draw[thick, opacity =.2] (0,.5) -- (1.5,-1);
		\draw[thick, opacity =.2] (0,0) -- (1,-1);
		\draw[thick, opacity =.2] (0,-.5) -- (.5,-1);
		\node at (2,-.5) (B) {1};
		\node at (1.5,-.5) (G) {1};
		\node at (1,0) (H) {1};
		\node at (.5,.5) (I) {1};
		\node at (0,1) (J) {1};
		\draw[->] (B) -- (G);
		\draw[->] (B) -- (H);
		\draw[->] (B) -- (I);
		\draw[->] (B) -- (J);
		\node at (1,-.5) (K) {1};
		\node at (.5,0) (L) {2};
		\node at (0,.5) (M) {3};
		\draw[->] (G) -- (K);
		\draw[->] (G) -- (L);
		\draw[->] (G) -- (M);
		\draw[->] (H) -- (K);
		\draw[->] (H) -- (L);
		\draw[->] (H) -- (M);
		\draw[->] (I) -- (L);
		\draw[->] (I) -- (M);
		\draw[->] (J) -- (M);
		\node at (.5,-.5) (N) {1};
		\node at (0,0) (O) {3};
		\draw[->] (K) -- (N);
		\draw[->] (K) -- (O);
		\draw[->] (L) -- (N);
		\draw[->] (L) -- (O);
		\draw[->] (M) -- (O);
		\node at (0,-.5) (P) {1};
		\draw[->] (N) -- (P);
		\draw[->] (O) -- (P);
		\end{tikzpicture}
	\end{center}
\end{definition}

Again, we can define $\partial_v\tilde{D}^{p,q}(A)$ for a $k$-module $A$ in the analogous way and obtain the following.

\begin{lemma}\label{twisted_boundary_representation}
The pair of functors
\[
\partial_v\tilde{D}^{p,q}: k\mbox{-mod} \lradjunction \tcplx: Z^v_{p,q-1}
\]
is an adjoint pair. The \emph{inclusion of vertical boundaries} $\partial_v\tilde{D}^{p,q}\mono \tilde{D}^{p,q}$ is the morphism representing the natural map $d_0\colon X_{p,q}\to Z^v_{p,q-1}(X)$.

\end{lemma}

\begin{remark}\label{twisted_boundary_properties}
	On bidegree $(p,q-1)$, the inclusion of vertical boundaries is defined by $y_{p,q-1}\mapsto d_0(x_{p,q})$. Hence it is clearly injective since it maps bijectively the bases in Definition \ref{twisted_boundary} to the second factors of the unions defining the bases in Definition \ref{twisted_disk}. Moreover, this observation also proves that the cokernel of $\partial_v\tilde{D}^{p,q}\mono \tilde{D}^{p,q}$ is $\partial_v\tilde{D}^{p,q+1}$.
\end{remark}

\begin{corollary}\label{twisted_representation_map}
	We have the following natural isomorphisms for any twisted complex $X$, $p\geq 0$, and $q\in\mathbb Z$:
	\begin{align*}
		\tcplx(\partial_v\tilde{D}^{p,q}\mono \tilde{D}^{p,q},X)&=\cplx(S^{q-1}\mono D^{q},X_{p,*}).
	\end{align*}
\end{corollary}

We need to know that, analogously to our previous model categories, the vertical boundary of the disc is actually acyclic, which requires more work than proving it for the disc itself. 

\begin{lemma}\label{twisted_boundary_trivial}
	For $p>0$, $H_*(\tot(\partial_v\tilde{D}^{p,q}))=0$.
\end{lemma}

\begin{proof}

	The vertical differential $d_0$ on $\partial_v\tilde{D}^{p,q}$ is given by
	\begin{align*}
	d_0d_{i_1}\cdots d_{i_n}(y_{p,q-1})&=
	\sum_{j=1}^n(-1)^j\sum_{i_{j,1}+i_{j,2}=i_j}d_{i_1}\cdots d_{i_{j-1}}d_{i_{j,1}}d_{i_{j,2}}d_{i_{j+1}}\cdots d_{i_n}(y_{p,q-1}),
	\end{align*}
	see Definition  \ref{twisted_disk} and recall that $d_0(y_{p,q-1})=0$.

	\bigskip
	We will see that, up to degree shift and change of sign in the differential, the chain complex $\partial_v\tilde{D}^{p,q}_{p-s,*}$, $s\geq 2$, is isomorphic to $C^*(\Delta^{s-2}, k)$, the coaugmented simplicial cochain complex of the indicated simplex with coefficients in our ground ring, which is contractible. 
	
	Recall that the augmented simplicial chain complex $C_*(\Delta^n,k)$, $n\geq 0$, is freely generated in each degree $-1\leq t\leq n$ by the strictly increasing sequences $[v_0,\dots,v_t]$ of length $n+1$ formed by integers $0\leq v_i\leq n$. The rank of $C_t(\Delta^n,k)$ is therefore $\binom{n+1}{t+1}$ for $-1\leq t\leq n$. The complex is zero elsewhere. Its differential is defined by
	\[d([v_0,\dots,v_t])=\sum_{i=0}^t(-1)^{i}[v_0,\dots,\widehat{v_i},\dots, v_t],\]
	where $\widehat{v_i}$ means $v_i$ removed. The cochain complex $C^*(\Delta^n,k)$ is the $k$-linear dual of $C_*(\Delta^n,k)$, hence it is freely generated by the dual basis elements $[v_0,\dots,v_t]^*$, whose differentials are
	\[d([v_0,\dots,v_t]^*)=\sum_{i=0}^{t+1}\sum_{v_{i-1}<u<v_{i}}(-1)^{i}[\dots,v_{i-1},u,v_{i},\dots]^*.\]
	Here, in the index of the second summation we understand that $v_{-1}=-1$ and $v_{t+1}=n+1$.
	
	As we have seen in Definition \ref{twisted_boundary}, the rank of $\partial_v\tilde{D}^{p,q}_{p-s,q-1+s-n}$ is $\binom{s-1}{n-1}$, the same as the rank of $C^{n-2}(\Delta^{s-2},k)$. Moreover, we have an isomorphism
	\[\partial_v\tilde{D}^{p,q}_{p-s,q-1+s-n}\cong C^{n-2}(\Delta^{s-2},k)\]
	given by the basis bijection
	\[d_{i_1}\cdots d_{i_n}(y_{p,q-1})\mapsto \bigg[ i_2-1, i_2+i_3-1, \cdots , \sum_{j=2}^ni_j-1\bigg]^*.\]
	This isomorphism is clearly compatible with differentials up to a $-1$ sign.
	
	Therefore, $$H^v_{p-s,*}(\partial_v\tilde{D}^{p,q})=0 \,\,\,\mbox{if}\,\,\, s\geq 2.$$ The remaining non-trivial part of $H^v(\partial_v\tilde{D}^{p,q})$ reduces to 
	\[
	H^v_{p,q-1}(\partial_v\tilde{D}^{p,q})=H^v_{p-1,q-1}(\partial_v\tilde{D}^{p,q})=k.
	\]
	 The differential $d_1$ on $H^v(\partial_v\tilde{D}^{p,q})$ is an isomorphism between these two modules, i.e.~$H^v(\partial_v\tilde{D}^{p,q})=\partial_vD^{p,q}$, hence $$H^h(H^v(\partial_v\tilde{D}^{p,q}))=0.$$
	  As this is the $E^2$-term of the spectral sequence in Remark \ref{twisted_properties}, this implies $$H_*(\tot(\partial_v\tilde{D}^{p,q}))=0$$ as claimed. 
\end{proof}

We note that Lemma \ref{twisted_disk_trivial} for $p>0$ could also have been derived from Lemma \ref{twisted_boundary_trivial} and Remark \ref{twisted_boundary_properties}. Furthermore, the isomorphism $$\partial_v\tilde{D}^{p,q}_{p-s,*} \cong C^*(\Delta^{s-2}, k),\quad s\geq 2,$$ in the previous proof is of course a great convenience, but we are not aware of a conceptual reason as to why simplicial cochains appear in this proof. 

We now consider the monoidal structure on twisted complexes, attributed by Meyer \cite{Meyer_1978} to Liulevicius \cite{liulevicius}.

\begin{definition}\label{monoidal_twisted_complex}
	The \emph{tensor product} $X\otimes Y$ of two twisted complexes $X$ and $Y$ is the twisted complex with the same underlying bigraded module as in Definition \ref{monoidal_bicomplex}, such that the differential on $\tot(X\otimes Y)=\tot(X)\otimes\tot(Y)$ is the standard tensor product differential. This is equivalent to saying that the maps $d_i$ on $X\otimes Y$ are given by
	\[d_i(x\otimes y)=d_i(x)\otimes y+(-1)^{\abs{x}}x\otimes d_i(y),\quad i\geq 0.\]
	
	This tensor product equips $\tcplx$ with a closed symmetric monoidal structure with obvious associativity and unit constraints. The unit and symmetry constraints are the same as in Definition \ref{monoidal_bicomplex}. The mapping objects are defined by the submodules
	\[\hom_{\tcplx}(X,Y)_{p,q}\subset \prod_{\substack{s\geq 0\\t\in\mathbb Z}}\hom_k(X_{s,t},Y_{s+p,t+q}),\qquad p\geq 0, \quad q\in\mathbb Z,\]
	formed by the elements $f$ such that, for $i>p$,
	\[d_if=(-1)^{\abs{f}}fd_i.\]
	The maps $d_i$ are defined by
	\[d_i(f)=d_if-(-1)^{\abs{f}}fd_i.\]
\end{definition}

The totalisation functor on twisted complexes is strong symmetric monoidal in the obvious naive way.

\bigskip
We will now construct the total model structure on $\tcplx$ by checking that the given weak equivalences, generating cofibrations, and generating trivial cofibrations satisfy \cite[Theorem 2.1.29]{hovey_model_1999}. 

\pagebreak
These are

\begin{align*}
\mathcal{W}& =\{ f \in \tcplx \,\,|\,\, \tot(f) \mbox{ is a quasi-isomorphism in $\cplx$} \}, \\
	I&=\{\partial_v\tilde{D}^{p,q}\mono \tilde{D}^{p,q}\}_{\substack{p\geq0\\q\in\mathbb Z}},\\
	J&=\{0\mono \tilde{D}^{0,q}\}_{q\in	\mathbb Z}\cup\{\partial_v\tilde{D}^{p,q}\mono \tilde{D}^{p,q}\}_{\substack{p>0\\q\in\mathbb Z}}.
	\end{align*}

\begin{theorem}\label{total_twisted_complexes}
	The category of twisted complexes $\tcplx$ can be equipped with a proper combinatorial abelian model category structure  such that:
	\begin{itemize}
	\item $f\colon X\to Y$ is a weak equivalence if $\tot(f)$ is a quasi-isomorphism in $\cplx$,
	\item  $f$ is a (trivial) fibration if it is pointwise surjective and $H^v_{p,q}(f)$ is an isomorphism for all $p>0$ (resp.~$p\geq 0$) and $q\in\mathbb{Z}$,
	\item $f\colon X\into Y$ is a cofibration in $\tcplx$ if and only if it is injective with cofibrant cokernel. Cofibrant implies degreewise projective.
	\end{itemize}
\end{theorem}

\begin{proof}
	This proof is essentially the same as that of Theorem \ref{total_bicomplexes}. 	
		The characterisation of (trivial) fibrations in the statement follows from Lemma \ref{twisted_representation_map} instead of Lemma \ref{representation}. (Trivial) fibrations and surjective weak equivalences can be detected by their kernels for exactly the same reason. 
		
A pushout along a generating cofibration adds copies of $k$ in certain degrees, see Remark \ref{twisted_boundary_properties}. Therefore, cofibrations are monomorphisms with pointwise projective cokernel. 
	
	Properties (4), (5), and (6) in \cite[Theorem 2.1.19]{hovey_model_1999} follow by the same arguments, using here the spectral sequence in Remark \ref{twisted_properties}, which is the twisted analog of that in Remark \ref{spectral_sequence}. 
	
	The argument for properness is literally the same. The fact that $\tot$ takes $I$ to cofibrations in $\cplx$ follows easily from Remark \ref{twisted_boundary_properties}.
	
\end{proof}

\begin{remark}
	As in the total model structure on bicomplexes, a twisted complex $X$ is fibrant in $\tcplx$ whenever its vertical homology is concentrated in horizontal degree $0$, and it is trivially fibrant if the vertical homology vanishes completely. 
\end{remark}

\begin{proposition}
The total model structure on $\tcplx$ is monoidal, has a cofibrant tensor unit, and satisfies the monoid axiom. 
\end{proposition}

\begin{proof}

As $\tcplx$ is an abelian model category, we can use
	 \cite[Theorem 4.2]{hovey_cotorsion_2007} to prove monoidality. Hypothesis (1) follows from the fact that cofibrant twisted complexes are pointwise projective. The tensor unit is cofibrant since it is the cokernel of the generating cofibration $\partial_v\tilde{D}^{0,0}\mono \tilde{D}^{0,0}$, which gives us (4). 
	 
	 \bigskip
	 With (2) and (3) we do not proceed in the same way as in the proof of Theorem \ref{total_bicomplexes}, since it would be even more complicated than what follows. By adjunction, the claims are equivalent to prove that, for any cofibrant twisted complex $X$ and any (trivial) fibration $f$, $\hom_{\tcplx}(X,f)$ is a (trivial) fibration of the mapping objects, and for any trivially cofibrant twisted complex $Y$ and any fibration $f$, $\hom_{\tcplx}(Y,f)$ is a trivial fibration, compare \cite[Lemma 4.2.2]{hovey_model_1999}. As remarked in the proof of Theorem \ref{total_bicomplexes}, it suffices to take $X$ to be the cokernel of a generating cofibration and to take $Y$ to be the cokernel of a generating trivial cofibration. This strategy has the advantage that we do not have to have an explicit characterisation of the cofibrations in $\tcplx$. 
	 
	 \bigskip	 
	 By Remark \ref{twisted_boundary_properties}, those cokernels are 
	 \[
	 X=\partial_v\tilde{D}^{0,q},  \,\,\,Y=\partial_v\tilde{D}^{p,q}, \,\,\,Y=\tilde{D}^{0,q},\qquad p>0,\quad q\in\mathbb Z.\]
	 We start with the two easy cases. 
	
	\bigskip
	The twisted complex $\partial_v\tilde{D}^{0,q}$ is $k$ concentrated in bidegree $(0,q-1)$, hence for any $A \in \tcplx$ we have natural isomorphisms
	\begin{align*}
	\hom_{\tcplx}(\partial_v\tilde{D}^{0,q},A)_{s,t}&=A_{s,t+q-1},\\
	H^v_{s,t}(\hom_{\tcplx}(\partial_v\tilde{D}^{0,q},A))&=H_{s,t+q-1}^v(A).
	\end{align*}
	The claim for $X=\partial_v\tilde{D}^{0,q}$ is an obvious consequence of these formulas as a map is a (trivial) fibration if and only if it is surjective and an isomorphism on $H^{v}_{p,q}$ for $p>0$ (resp.~$p \ge 0$). 
	
	\bigskip
	 Next, since $\tilde{D}^{0,q}$ is $D^q$ concentrated in horizontal degree $0$, we have natural isomorphisms of chain complexes, $s\geq 0$,
	\[\hom_{\tcplx}(\tilde{D}^{0,q}, A)_{s,*}=\hom_{\cplx}({D}^{q},A_{s,*}),\]
	the first having differential $d_0$. The object ${D}^{q}$ is trivially cofibrant in $\cplx$ and hence mapping out of it preserves (trivial) fibrations in $\cplx$. If $f\colon A \rightarrow B$ is a fibration in $\tcplx$ then $$f_{s,\ast}: A_{s,\ast} \longrightarrow B_{s,\ast}$$ is a fibration of chain complexes, which implies that $\hom_{\cplx}({D}^{q},f_{s,*})$ is a trivial fibration of complexes for all $s\geq 0$. Hence $\hom_{\tcplx}(\tilde{D}^{0,q}, A)$ is a trivial fibration in $\tcplx$.

\bigskip
	Let $p>0$. We now consider the most difficult case, $Y=\partial_v\tilde{D}^{p,q}$, which requires more calculations. Recall the basis of the underlying bigraded module of $\partial_v\tilde{D}^{p,q}$ in Definition \ref{twisted_boundary}. For each $s\geq 0$, we consider the bigraded submodule $$\partial_v\tilde{D}^{p,q,s}\subset \partial_v\tilde{D}^{p,q}$$ generated by the elements $y_{p,q-1}$ and $d_{i_1}\cdots d_{i_n}(y_{p,q-1})$ with $i_1\leq s$. The inclusion $\partial_v\tilde{D}^{p,q,s}\subset \partial_v\tilde{D}^{p,q}$ is compatible with the vertical differential $d_0$, see the formula at the beginning of the proof of Lemma \ref{twisted_boundary_trivial}. 
		
	Using this basis and the definition of the mapping object $\tcplx$ we see that, for all $Z \in \tcplx$, $s\geq 0$, and $t\in\mathbb Z$, the composite
	\[
	\hom_{\tcplx}(\partial_v\tilde{D}^{p,q},Z)_{s,t}\subset \prod_{\substack{u\geq 0\\w\in\mathbb Z}}\hom_k(\partial_v\tilde{D}^{p,q}_{u,w},Z_{u+s,w+t})
	\onto
	\prod_{\substack{u\geq 0\\w\in\mathbb Z}}\hom_k(\partial_v\tilde{D}^{p,q,s}_{u,w},Z_{u+s,w+t})
	\]
	is an isomorphism. In particular, $\hom_{\tcplx}(\partial_v\tilde{D}^{p,q},-)$ preserves surjections since $\partial_v\tilde{D}^{p,q,s}$ is pointwise free. 
	Our model structure is abelian, which means that (trivial) fibrations are exactly the surjections with (trivially) fibrant kernel.
	Hence, it suffices to prove that $\hom_{\tcplx}(\partial_v\tilde{D}^{p,q},Z)$ is trivially fibrant for any fibrant twisted complex $Z$.

	\bigskip
	A twisted complex is trivially fibrant if it has trivial vertical homology. For any $s\geq 0$, the previous isomorphism yields an identification of chain complexes
		\[\hom_{\tcplx}(\partial_v\tilde{D}^{p,q},Z)_{s,*}
		=
		\prod_{u\geq 0}
		\hom_{\cplx}(\partial_v\tilde{D}^{p,q,s}_{u,*},Z_{u+s,*}),
		\]
	the first having differential $d_0$ again. 
	The twisted complex $\partial_v\tilde{D}^{p,q}$ is concentrated in horizontal degrees $\leq p$. Hence the previous product is actually indexed by $0\leq u\leq p$.
	We must prove that each factor of the product is acyclic. We distinguish the possible cases.

	\pagebreak
	\fbox{$s=0$}
	\medskip
	
	If $s=0$, $\partial_v\tilde{D}^{p,q,0}_{p,*}$ is $k$ concentrated in vertical degree $q-1$ and $\partial_v\tilde{D}^{p,q,0}_{u,*}=0$ for $u\neq p$. In the latter case, $$\hom_{\cplx}(\partial_v\tilde{D}^{p,q,0}_{u,*},Z_{u,*})=0,$$ and in the former,
	\[H_t(\hom_{\cplx}(\partial_v\tilde{D}^{p,q,0}_{p,*},Z_{p,*}))
	=H_{t+q-1}(Z_{p,*})=H^v_{p,t+q-1}(Z)=0.
	\]
	This is indeed zero since $p>0$ and $Z$ is a fibrant twisted complex.

	\bigskip
		\fbox{$s>0,\,\, u=p, p-1$}
	\medskip
		
	% From now on, let $s>0$. 
	Both $\partial_v\tilde{D}^{p,q,s}_{p,*}$ and $\partial_v\tilde{D}^{p,q,s}_{p-1,*}$ are $k$ concentrated in vertical degree $q-1$, son we can check as right above that $\hom_{\cplx}(\partial_v\tilde{D}^{p,q,s}_{u,*},Z_{u+s,*})$ is acyclic. % for $u=p,p-1$.

	\bigskip
	\fbox{$s>0,\,\, p-s \leq u \leq p-2$}
	\medskip
	
	The restriction $i_1\leq s$ is empty in horizontal degrees $\geq p-s$ in $\partial_v\tilde{D}^{p,q}$ because in bidegree $(p-j, q-1+j-n)$ it is generated by $d_{i_1} \cdots d_{i_n} (y_{p,q-1})$ with $i_1, \cdots i_n >0$ and $i_1 + \cdots + i_n = j$.
	Hence,  
	\[
	\partial_v\tilde{D}^{p,q,s}_{u,*}=\partial_v\tilde{D}^{p,q}_{u,*} \,\,\,\mbox{for}\,\,\, u\geq p-s.
	\] 
	In the proof of Lemma \ref{twisted_boundary_trivial}, for $u\leq p-2$ we have established an identification between $\partial_v\tilde{D}^{p,q}_{u,*}$ and the coaugmented cochain complex of a simplex $C^*(\Delta^{p-u-2}, k)$, which is trivially cofibrant in $\cplx$. Therefore $\hom_{\cplx}(\partial_v\tilde{D}^{p,q,s}_{u,*},Z_{u+s,*})$ is also acyclic for $p-s\leq u\leq p-2$.

	\bigskip
	\fbox{$s>0,\,\, u=p-s-1$}
	\medskip
		
	This is the first case where the condition $i_1\leq s$ is meaningful, but it only discards $d_{s+1}(y_{p,q-1})$, the topmost non-trivial free generator of the trivially cofibrant complex $\partial_v\tilde{D}^{p,q}_{p-s-1,*}$. Therefore, $\partial_v\tilde{D}^{p,q,s}_{p-s-1,*}$ contains $k$ concentrated in vertical degree $q+s-2$ as a strong deformation retract. Hence,
	\[
	H_t(\hom_{\cplx}(\partial_v\tilde{D}^{p,q,s}_{p-s-1,*},Z_{p-1,*}))=H_{t+q+s-2}(Z_{p-1,*})=H^v_{p-1,t+q+s-2}(Z)=0,
	\] 
	since $Z$ is fibrant and $u\geq 0$, so $p\geq s+1>1$, i.e.~$p-1>0$.

	\bigskip
	\fbox{$s>0,\,\, 0 \leq u < p-s-1$}
	\medskip
	
	Using the explicit identification of $\partial_v\tilde{D}^{p,q}_{u,*}$ with $C^*(\Delta^{p-u-2}, k)$ in the proof of Lemma \ref{twisted_boundary_trivial}, we can identify the vertical subcomplexes 
	\[
	\partial_v\tilde{D}^{p,q,s}_{u,*} \cong C^*(\Delta^{p-u-2}, \Delta^{p-u-s-2}, k).
	\]
	This follows from the fact that, by the constraint $\sum_{j=1}^ni_j=p-u$ for the elements $d_{i_1} \cdots d_{i_n} (y_{p,q-1})$ of the basis of $\partial_v\tilde{D}^{p,q,s}_{u,*}$, the condition $i_1\leq s$ is equivalent to
	\[\sum_{j=2}^ni_2-1\geq p-u-s-1.\]
	Hence the basis elements of $\partial_v\tilde{D}^{p,q,s}_{u,*}$ correspond to the duals of the simplices in $\Delta^{p-u-2}$ which are not in $\Delta^{p-u-s-2}$. The latter simplex makes sense by the upper bound of $u$.
	
	This shows that $\partial_v\tilde{D}^{p,q,s}_{u,*}$ is trivially cofibrant under the current hypotheses, so $\hom_{\cplx}(\partial_v\tilde{D}^{p,q,s}_{u,*},Z_{u+s,*})$ is also acyclic. Therefore, we have finally proved that $\tcplx$ is monoidal.
	
	\bigskip
	Lastly, the functor $\tot$ takes generating trivial cofibrations in $\tcplx$ to cofibrations in $\cplx$, compare Remark \ref{twisted_boundary_properties}, which are trivial by Lemmas \ref{twisted_disk_trivial} and \ref{twisted_boundary_trivial}. Hence the monoid axiom for $\tcplx$ follows from the validity of the monoid axiom in $\cplx$.

\end{proof}

\begin{proposition}\label{quillen_equivalence_twisted_complexes}
	The inclusion of chain complexes as twisted complexes concentrated in horizontal degree $0$ is the left adjoint of a strong symmetric monoidal Quillen equivalence,
	\[\cplx\rightleftarrows\tcplx.\]
\end{proposition}

\begin{proof}
	As in the proof of Proposition \ref{quillen_equivalence_bicomplexes_complexes}, the left adjoint obviously preserves the tensor product and the tensor unit, and the right adjoint is $$\tcplx\to\cplx\colon X\mapsto X_{0,*}.$$ Clearly, this right adjoint preserves (trivial) fibrations. This shows that the adjoint pair is a Quillen pair. 
	
	The same argument as in Proposition \ref{quillen_equivalence_bicomplexes_complexes} shows that it is a Quillen equivalence. Here we should use the spectral sequence in Remark \ref{twisted_properties} instead of Remark \ref{spectral_sequence}.
\end{proof}

\begin{corollary}
	The adjoint pair in Remark \ref{twisted_properties} defines a strong symmetric monoidal Quillen equivalence \[\tcplx\rightleftarrows\bcplx_{\tot}.\]
\end{corollary}

\begin{proof}
	The right adjoint, which is the full inclusion $\bcplx\subset\tcplx$, obviously preserves (trivial) fibrations, so the adjoint pair in the statement is a Quillen pair. This Quillen pair and the Quillen equivalence in Proposition \ref{quillen_equivalence_twisted_complexes}
	\[\cplx\rightleftarrows\tcplx\rightleftarrows\bcplx_{\tot}\]
	compose to the Quillen equivalence in Proposition \ref{quillen_equivalence_bicomplexes_complexes}. Hence the corollary follows from the $2$-out-of-$3$ property for Quillen equivalences \cite[Corollary 1.3.15]{hovey_model_1999}.
\end{proof}

% ----------------------------------------------------------------

\providecommand{\bysame}{\leavevmode\hbox to3em{\hrulefill}\thinspace}
\providecommand{\MR}{\relax\ifhmode\unskip\space\fi MR }
\providecommand{\MRhref}[2]{%
  \href{http://www.ams.org/mathscinet-getitem?mr=#1}{#2}
}
\providecommand{\href}[2]{#2}

\end{document}